\documentclass[12pt]{amsart}
\pdfoutput=1
\usepackage[margin=1in]{geometry}
\usepackage[utf8]{inputenc}
\usepackage{graphicx, cite}
\usepackage{amsmath,amssymb}
\usepackage{amsthm}
\usepackage{color}
\usepackage{tikz}
\usepackage[utf8]{inputenc}
\usepackage{pgfplots}
\pgfplotsset{compat=1.18}
\usetikzlibrary{fillbetween}
\allowdisplaybreaks

\makeatletter
\def\l@section{\@tocline{1}{12pt plus2pt}{0pt}{}{\bfseries}}
\def\l@subsection{\@tocline{2}{0pt}{2pc}{2pc}{}}
\makeatother
\makeatletter
\def\subsection{\@startsection{subsection}{2}{\z@}%
	{-3.25ex\@plus -1ex \@minus -.2ex}%
	{1.5ex \@plus .2ex}%
	{\normalfont\bfseries\boldmath}}
\def\subsubsection{\@startsection{subsubsection}{3}%
	\z@{.5\linespacing\@plus.7\linespacing}{-.5em}%
	{\normalfont\bfseries\boldmath}}
\renewcommand\paragraph{\@startsection{paragraph}{4}{\z@}%
	{3.25ex \@plus1ex \@minus.2ex}%
	{-1em}%
	{\normalfont\normalsize\bfseries}}
\makeatother

\theoremstyle{plain}
\newtheorem{thm}{Theorem}[section]

\newtheorem{cor}[thm]{Corollary}
\newtheorem{lem}[thm]{Lemma}
\newtheorem{prop}[thm]{Proposition}

\theoremstyle{definition}
\newtheorem{defn}[thm]{Definition}

\theoremstyle{remark}
\newtheorem{rem}[thm]{Remark}

\theoremstyle{plain}

\numberwithin{equation}{section}

\theoremstyle{plain} 
\newcommand{\thistheoremname}{}
\newtheorem{genericthm}[thm]{\thistheoremname}

  \newtheorem*{genericthm*}{\thistheoremname}
\newenvironment{namedthm*}[1]
  {\renewcommand{\thistheoremname}{#1}%
   \begin{genericthm*}}
  {\end{genericthm*}}

\newcommand{\diag}{\textnormal{diag}}

\newcommand{\D}{{\mathbb D}}

\newcommand{\C}{{\mathbb C}}

\newcommand{\N}{{\mathbb N}}

\newcommand{\calD}{{\mathcal D}}

\newcommand{\calQ}{{\mathcal Q}}

\newcommand{\calB}{{\mathcal B}}

\newcommand{\supp}{{\textnormal{supp}}}

\makeatletter
\newcommand{\vast}{\bBigg@{4}}
\newcommand{\Vast}{\bBigg@{5}}
\makeatother

\def\udot#1{\ifmmode\oalign{$#1$\crcr\hidewidth.\hidewidth
    }\else\oalign{#1\crcr\hidewidth.\hidewidth}\fi}

\def\T{\mathbb{T}}
\def\C{\mathbb{C}}

\def\beq{\begin{equation}}
\def\eeq{\end{equation}}
\makeatletter
\newcommand{\doublewidetilde}[1]{{%
  \mathpalette\double@widetilde{#1}%
}}
\newcommand{\double@widetilde}[2]{%
  \sbox\z@{$\m@th#1\widetilde{#2}$}%
  \ht\z@=.9\ht\z@
  \widetilde{\box\z@}%
}
\makeatother

\def\one{\mbox{1\hspace{-4.25pt}\fontsize{12}{14.4}\selectfont\textrm{1}}}

\usepackage{tikz}

\makeatletter
\def\@makefnmark{%
  \leavevmode
  \raise.9ex\hbox{\fontsize\sf@size\z@\normalfont\tiny\@thefnmark}}
\makeatother

\begin{document}

\title[Carleson tents and Bloch--Carleson measures]{The $p$--Bloch--Carleson measure problem for $0<p<\infty$: A complete characterization via Carleson tents}

\author{Bingyang Hu}
\address{(Bingyang Hu) Department of Mathematics and Statistics\\
         Auburn University\\
         Auburn, Alabama, U.S.A, 36849}
\email{bzh0108@auburn.edu}

\author{Xiaojing Zhou}
\address{(Xiaojing Zhou) Department of Mathematics and Statistics\\
         Auburn University\\
         Auburn, Alabama, U.S.A, 36849}
\email{xiz0003@auburn.edu}
\date{October 25, 2026}


\begin{abstract}
We characterize the boundedness and compactness of the embedding
$$
\operatorname{id}:\calB\longrightarrow L^p(\mu),\qquad 0<p<\infty,
$$
where $\calB$ denotes the Bloch space and $\mu$ is a positive Borel measure on the unit disc. The criteria depend only on the  dyadic Carleson tents. These results refine our earlier $L^2(\mu)$ characterization based on a full dyadic resolution of the disc and extend it to all positive exponents. The proof combines discretization on Carleson tents, representations of Bloch functions by series of analytic polynomials, and weighted matrix estimates. We also obtain corresponding boundedness and compactness characterizations for
$$
\operatorname{id}:\calQ_p\longrightarrow L^q(\mu),\qquad 0<p\le1,\quad 0<q<\infty.
$$
These results resolve the general $p$--Bloch--Carleson measure problem studied by Girela, Pel\'aez, P\'erez-Gonz\'alez and R\"atty\"a in 2008.
\end{abstract}
\maketitle
\tableofcontents

\section{Introduction}

Let $\D$ be the unit disc, $dA(z)=(1/\pi)dxdy$ be the normalized area measure, $\T=\partial\D$, and $H(\D)$ denote the space of holomorphic functions on $\D$ equipped with the compact--open topology. The Bloch space
$\calB$ consists of all $f\in H(\D)$ such that
$$
\|f\|_{\calB}:=|f(0)|+\|f\|_{\calB,*}<\infty,
\qquad
\|f\|_{\calB,*}:=\sup_{z\in\D}(1-|z|^2)|f'(z)|.
$$
The \emph{Bloch--Carleson measure problem} asks for necessary and sufficient
conditions on a positive Borel measure $\mu$ such that the embedding
\begin{equation}\label{20260614eq01}
\operatorname{id}:\calB \longrightarrow L^2(\mu)
\end{equation}
is bounded or compact, where $L^2(\mu)$ denotes the weighted $L^2$-space
associated with $\mu$. More generally, one may ask the same question for
the embedding
\begin{equation}\label{20261015eq01}
\operatorname{id}:\calB\longrightarrow L^p(\mu),\qquad 0<p<\infty.
\end{equation}
We call $\mu$ a \emph{$p$-Bloch--Carleson measure} if the embedding
\eqref{20261015eq01} is bounded.

\vspace{0.1cm}

The Bloch--Carleson measure problem has a long history. Girela,
Pel\'aez, P\'erez-Gonz\'alez and R\"atty\"a \cite{GPPR} proved that
\begin{equation} \label{20261016cond01}
\int_{\D}\left(\log\frac{e}{1-|z|}\right)^p\,d\mu(z)<\infty
\end{equation} 
is sufficient for \eqref{20261015eq01}, whereas
\begin{equation} \label{20261016cond02}
\int_{\D}\left(\log\frac{e}{1-|z|}\right)^{p/2}\,d\mu(z)<\infty
\end{equation} 
is necessary. Neither condition characterizes arbitrary
$p$-Bloch--Carleson measures, although complete characterizations
are available for several special classes of measures. Bao, Du,
Wulan and Zhu \cite{BDWZ} obtained further equivalent descriptions
using weighted Bergman spaces associated with Green-type potentials.

In our earlier work \cite{HZBloch}, we characterized the boundedness
and compactness of the embedding \eqref{20260614eq01} by a dyadic
matrix capacity associated with an admissible dyadic resolution
of the whole unit disc. One example of such a dyadic resolution
is given in \cite[Section~2.2]{HZBloch}. In that construction,
the standard dyadic system in $[0,1)^2$ is transformed to $\D$
through the map
$$
\Theta:[0,1)^2\longrightarrow\D,
\qquad
\Theta(s,t)=\sqrt{s}\,e^{2\pi it}.
$$
In particular, for each $N\ge1$, the $N$--th generation of this
admissible dyadic resolution consists of the tiles
$$
\Theta\left(
\left[\frac{j}{2^N},\frac{j+1}{2^N}\right)
\times
\left[\frac{k}{2^N},\frac{k+1}{2^N}\right)
\right),
\qquad 0\le j,k<2^N.
$$

\begin{rem} \label{20261015rem01}
Although the above approach gives a complete characterization for $p=2$,
it leaves two natural issues to be addressed.

First, classical Carleson measure criteria for analytic function
spaces are typically formulated in terms of \emph{Carleson tents}.
The full dyadic resolution above has a different geometry:
both the radial and angular variables are refined at every
generation, even in regions that remain a fixed positive distance
from the boundary. Such continued subdivision of the interior
does \emph{not} seem intrinsic to the embedding problem in the complex setting. Indeed, Cauchy's integral formula recovers the values of a holomorphic
function in an interior region from its values on a surrounding
contour, while the Bloch norm already gives uniform control on
each fixed compact subset of $\D$. From this viewpoint, it is
natural to seek a characterization adapted to the distance from
the boundary, as in the geometry of Carleson tents. This leads
to the following question: \textit{Can the boundedness and compactness of the Bloch embedding be characterized solely in terms of the Carleson tents?} 

Second, the matrix argument in \cite{HZBloch} relies essentially
on the quadratic structure of $L^2(\mu)$.  This is the step that permits the use of
Grothendieck's inequality and semidefinite duality. For $p\ne2$, the corresponding $L^p(\mu)$ integral no longer has
this quadratic form, so the argument does \emph{not} directly yield
a characterization in the general case.
\end{rem}

The \emph{purpose} of the present paper is to give a complete
characterization of the boundedness and compactness of the embedding
\eqref{20261015eq01} for every $0<p<\infty$, in terms of a new dyadic
capacity determined solely by the system of dyadic Carleson tents.
This addresses both issues raised in Remark~\ref{20261015rem01}.
The significance of this characterization is twofold:
\begin{enumerate}
\item
It refines and extends the results of our earlier work
\cite{HZBloch}, replacing the full dyadic resolution of $\D$
by the classical geometry of Carleson tents and extending the
$L^2(\mu)$ characterization to the full range $0<p<\infty$.

\item
It provides a complete solution to the general $p$--Bloch--Carleson measure problem studied by Girela,
Pel\'aez, P\'erez-Gonz\'alez and R\"atty\"a in \cite{GPPR}, and subsequently by Bao, Du, Wulan and Zhu in \cite{BDWZ}.
\end{enumerate}

We now turn to the statements of our main results. Observe first that since constant
functions belong to $\calB$, the condition $\mu(\D)<\infty$ is
necessary for the boundedness of \eqref{20261015eq01}.
Under this assumption, the boundedness problem reduces to a
seminorm estimate. To formulate this estimate, set
$$
\mathring{\calB}:=\{f\in\calB:f(0)=0\}.
$$
It remains to characterize those finite positive Borel measures
$\mu$ for which
\begin{equation} \label{20261015eq01x}
\int_\D |f(z)|^p\,d\mu(z)
\lesssim \|f\|_{\calB,*}^p,
\qquad f\in\mathring{\calB}.
\end{equation} 

\subsection{Carleson tents and the associated matrices}

We start with some basic definitions. For an arc $I\subseteq\T$, let $|I|$ denote its normalized arclength.
The \emph{Carleson tent} over $I$ and its \emph{upper half} are defined by
$$
\begin{aligned}
Q_I&:=\{re^{it}\in\D:e^{it}\in I,\ 1-|I|\le r<1\},\\
Q_I^{\mathrm{up}}&:=\{re^{it}\in\D:e^{it}\in I,
\ 1-|I|\le r<1-|I|/2\}.
\end{aligned}
$$
Let $\calD$ be a dyadic system on $\T$, with
$\calD_n$ denoting its $n$th generation. The corresponding upper
tents form a partition of $\D$, that is, 
\begin{equation}\label{eq:partition}
\D=\bigsqcup_{I\in\calD}Q_I^{\mathrm{up}}.
\end{equation}
For each arc $I\ne\T$, let $\zeta_I$ denote its midpoint and set
\begin{equation} \label{20261016defn01X}
z_I=(1-\tfrac34|I|)\zeta_I,
\qquad z_{\T}=0.
\end{equation}

Next, we define the associated dyadic gauge matrices induced by any dyadic system $\calD$ on $\T$. Fix a nonnegative smooth function
\begin{equation}\label{20261017eq12}
\eta\in C_c^\infty((1/16,3/4)),
\qquad \textrm{with} \quad  \eta(s)=1 \quad   \textrm{when} \quad  1/8\le s\le1/2,
\end{equation}
and define, for $I\in\calD$,
\begin{equation}\label{20261017eq05}
\Phi_I(z):=|I|\sum_{k\ge1}\eta(k|I|)
(\overline{\zeta_I}z)^k,\qquad z\in\D.
\end{equation}
It is clear that each $\Phi_I$ is an analytic polynomial with $\Phi_I(0)=0$,
and $\Phi_{\T}=0$. Moreover, since the sum contains $\lesssim |I|^{-1}$
nonzero terms,
\begin{equation}\label{eq:individual-bound}
\sup_{I\in\calD}\sup_{z\in\D}|\Phi_I(z)|\lesssim1.
\end{equation}

Write $\calD_{\le N}:=\bigcup_{n=0}^N\calD_n$. For nonnegative
weights $h=(h_K)_{K\in\calD}$, define the \emph{$N$-th dyadic gauge matrix} associated to $\calD$ by 
\begin{equation} \label{20261023eq34}
\Gamma_{\mu,h,N}(I,J)
:=\sum_{K\in\calD}\mu(Q_K^{\mathrm{up}})h_K
\overline{\Phi_I(z_K)}\Phi_J(z_K),
\qquad I,J\in\calD_{\le N}.
\end{equation}
Here, the indices of each matrix are ordered
first by generation and then by angle. Note that if the diagonal entries of $\Gamma_{\mu,h,N}$ are finite,
then all its entries are finite by Cauchy--Schwarz. In this case, we record the following basic properties:
\begin{enumerate}
    \item $\Gamma_{\mu,h,N}$ is a complex matrix of dimension
    $\#\calD_{\le N}\times\#\calD_{\le N}$, where
    $\#\calD_{\le N}$ denotes the cardinality of $\calD_{\le N}$;

    \item $\Gamma_{\mu,h,N}$ is Hermitian, since
    $$
    \Gamma_{\mu,h,N}(I, J)
    =\overline{\Gamma_{\mu,h,N}(J, I)},
    \qquad I,J\in\calD_{\le N};
    $$

    \item $\Gamma_{\mu,h,N}$ is positive semidefinite.
    Indeed, for any $c\in\C^{\#\calD_{\le N}}$, we have
    \begin{equation}\label{eq:gram-form}
    \begin{aligned}
    c^*\Gamma_{\mu,h,N}c
    &=\sum_{I,J\in\calD_{\le N}}
      \overline{c_I}c_J\Gamma_{\mu,h,N}(I,J)\\
    &=\sum_{K\in\calD}\mu(Q_K^{\mathrm{up}})h_K
      \sum_{I,J\in\calD_{\le N}}
      \overline{c_I\Phi_I(z_K)}\,c_J\Phi_J(z_K)\\
    &=\sum_{K\in\calD}\mu(Q_K^{\mathrm{up}})h_K
      \left|\sum_{I\in\calD_{\le N}}c_I\Phi_I(z_K)\right|^2
      \ge0.
    \end{aligned}
    \end{equation}
    Here and throughout, for any complex vector $c$, we denote
    by $c^*$ its conjugate transpose.
\end{enumerate} For the constant weight
$h_K=1$, we write
\begin{equation}\label{eq:gram}
\Gamma_{\mu,N}(I,J)
:=\sum_{K\in\calD}\mu(Q_K^{\mathrm{up}})
\overline{\Phi_I(z_K)}\Phi_J(z_K).
\end{equation}
This matrix always has finite entries by \eqref{eq:individual-bound} and the fact that $\sum_{K \in \calD} \mu \left(Q_K^{\textrm{up}} \right)=\mu(\D)<+\infty$. Moreover, given $J\in\N$ and a finite nonnegative sequence $\lambda=\{\lambda_j\}_{1\le j\le J}$, we denote by 
$$ 
\diag(\lambda) = \diag(\lambda_1,\ldots,\lambda_J) 
$$ 
the $J\times J$ diagonal matrix whose $j$-th diagonal entry is $\lambda_j$.

\begin{defn}
Let $\mu$ be a finite positive Borel measure on $\D$.
We define its \emph{$p$-Carleson Bloch capacity}
$\calB_{\calD,p}(\mu)$ as follows.
\begin{enumerate} 
\item [$\bullet$] For $0<p<2$, set
\begin{equation}\label{20261016defn01}
\calB_{\calD,p}(\mu):=
\sup_{N\ge0}\ \inf_{\substack{h_K>0\\
\sum_{K\in\calD}\mu(Q_K^{\mathrm{up}})h_K^{\frac{p}{p-2}}\le1}}
\inf\left\{ \sum_{I\in\calD_{\le N}}\lambda_I: \lambda_I\ge0,\quad
\Gamma_{\mu,h,N}\le\operatorname{diag}(\lambda_I) \right\}.
\end{equation}
\item [$\bullet$] For $p=2$, set
\begin{equation}\label{20261016defn02}
\calB_{\calD,2}(\mu):=
\sup_{N\ge0}
\inf\left\{
\sum_{I\in\calD_{\le N}}\lambda_I:
\lambda_I\ge0,\quad
\Gamma_{\mu,N}\le\operatorname{diag}(\lambda_I)
\right\}.
\end{equation}
\item [$\bullet$] For $2<p<\infty$, set
\begin{equation}\label{20261016defn03}
\calB_{\calD,p}(\mu):=
\sup_{N\ge0}\ 
\sup_{\substack{h_K\ge0\\
\sum_{K\in\calD}\mu(Q_K^{\mathrm{up}})h_K^{\frac{p}{p-2}}\le1}}
\inf\left\{
\sum_{I\in\calD_{\le N}}\lambda_I:
\lambda_I\ge0,\quad
\Gamma_{\mu,h,N}\le\operatorname{diag}(\lambda_I)
\right\}.
\end{equation}
\end{enumerate} 
Here the matrix inequality $A\le B$ means that $B-A$ is
positive semidefinite. Moreover, if a diagonal entry of $\Gamma_{\mu,h,N}$
is infinite for some $h$ and $N$, the corresponding innermost infimum is understood
to be $+\infty$.
\end{defn}

\begin{rem}
We make the following remarks on the above definition.
\begin{enumerate}
\item For $p>2$, let $h=(h_K)_{K \in \calD}$ satisfy the constraints in
\eqref{20261016defn03}. Then by H\"older's, we have 
$$
\sum_{K\in\calD}\mu(Q_K^{\mathrm{up}})h_K \le \left( \sum_{K\in\calD}\mu(Q_K^{\mathrm{up}})
h_K^{p/(p-2)} \right)^{(p-2)/p}
\mu(\D)^{2/p} \le \mu(\D)^{2/p}<\infty.
$$
It follows from \eqref{eq:individual-bound} that all entries of $\Gamma_{\mu,h,N}$ are finite for every $N\ge0$.

\vspace{0.1cm}

\item For $0<p<2$, observe that a sequence $h=(h_K)_{K\in\calD}$ satisfying
the constraints in \eqref{20261016defn01} need \emph{not} yield
finite entries in $\Gamma_{\mu,h,N}$, since $p/(p-2)<0$. Nevertheless, we claim
that for each fixed $N\ge0$,
\begin{equation}\label{20261016claim01}
\inf_{\substack{h_K>0\\
\sum_{K\in\calD}\mu(Q_K^{\mathrm{up}})
h_K^{p/(p-2)}\le1}}
\inf\left\{
\sum_{I\in\calD_{\le N}}\lambda_I:
\lambda_I\ge0,\quad
\Gamma_{\mu,h,N}\le\operatorname{diag}(\lambda_I)
\right\}<\infty.
\end{equation}
To see this, we may assume that $\mu(\D)>0$, otherwise the
embedding \eqref{20261015eq01} is trivial when $\mu(\D)=0$.
Taking
\begin{equation}\label{20261016eq01}
h_K=\mu(\D)^{(2-p)/p},
\qquad K\in\calD,
\end{equation}
we obtain
$$
\sum_{K\in\calD}\mu(Q_K^{\mathrm{up}})h_K^{p/(p-2)}
=\frac{1}{\mu(\D)}
\sum_{K\in\calD}\mu(Q_K^{\mathrm{up}})=1.
$$
Thus this choice satisfies the required constraints. Moreover,
by \eqref{eq:individual-bound},
$$
\begin{aligned}
&\sum_{K\in\calD}\mu(Q_K^{\mathrm{up}})h_K
\left|\Phi_I(z_K)\Phi_J(z_K)\right|\\
&\qquad\lesssim
\mu(\D)^{(2-p)/p}
\sum_{K\in\calD}\mu(Q_K^{\mathrm{up}})
=\mu(\D)^{2/p}<\infty.
\end{aligned}
$$
Hence all entries of $\Gamma_{\mu,h,N}$ are finite for the
sequence $h$ defined by \eqref{20261016eq01}.
This proves \eqref{20261016claim01}.

\vspace{0.1cm}

    \item Finally, we show that the definitions of $\calB_{\calD, p}$
for $0<p<2$ and $p>2$ are consistent at the threshold $p=2$. First, we consider the case when $p>2$. Suppose that $\mu(\D)>0$.
For any fixed nonnegative sequence $h=(h_K)_{K \in \calD}$,
$$
\lim_{p \to 2^{+}} \left(
\sum_{K\in\calD}\mu(Q_K^{\mathrm{up}}) h_K^{p/(p-2)} \right)^{(p-2)/p}
=\sup_{\mu(Q_K^{\mathrm{up}})>0}h_K.
$$
Thus the constraint in
\eqref{20261016defn03} becomes $0\le h_K\le1$ when $p \to 2^{+}$ and $\mu(Q_K^{\textrm{up}})>0$, and for such a choice of nonnegative sequence $h=(h_K)_{K \in \calD}$, it is easy to see that 
    $$
    \Gamma_{\mu,h,N}\le\Gamma_{\mu,N},
    $$
    so the supremum over $h=(h_K)_{K \in \calD}$ is attained at $h_K=1$.

\vspace{0.1cm}

Next, we treat the case when $0<p<2$. Similarly, in this case we have 
$$
\lim_{p \to 2^{-}}
\left( \sum_{K\in\calD}\mu(Q_K^{\mathrm{up}})
h_K^{p/(p-2)} \right)^{(2-p)/p}
=\sup_{\mu(Q_K^{\mathrm{up}})>0}h_K^{-1}, 
$$
which means the constraint in \eqref{20261016defn01} becomes $h_K\ge1$ when $p \to 2^{-}$ and $\mu(Q_K^{\textrm{up}})>0$. 
For such a positive sequence $h=(h_K)_{K \in \calD}$, we now have 
    $$
    \Gamma_{\mu,h,N}\ge\Gamma_{\mu,N},
    $$
    so the infimum over $h=(h_K)_{K \in \calD}$ is attained at $h_K=1$.

    \vspace{0.1cm}
    
    Therefore, both endpoint formulations give
    \eqref{20261016defn02}. 
\end{enumerate}
\end{rem}

\subsection{Main results}

We are now ready to state the first main result of the paper. 

\begin{thm}[Boundedness]\label{boundedness}
Let $0<p<\infty$, $\calD$ be any dyadic system on $\T$, and $\mu$ be a finite positive Borel measure on
$\D$. Then
\eqref{20261015eq01} is bounded if and only if $\calB_{\calD,p}(\mu)<\infty$.
Moreover,
\begin{equation}\label{20261016eq05}
\|\operatorname{id}:\calB\to L^p(\mu)\|
\simeq\mu(\D)^{1/p}+\calB^{1/2}_{\calD,p}(\mu).
\end{equation}
The implicit constants depend only on $p$ and the cutoff function $\eta$,
and are independent of $\mu$ and $\calD$.
\end{thm}

\begin{rem}\label{rem:mass-term}
The term $\mu(\D)^{1/p}$ in \eqref{20261016eq05} is necessary. Indeed,
if $\mu$ is supported in $Q_{\T}^{\mathrm{up}}$, then
$\calB_{\calD,p}(\mu)=0$, because $z_{\T}=0$ and $\Phi_I(0)=0$.
On the other hand, by testing the constant function $1$, one has $\mu(\D) \lesssim \|\operatorname{id}:\calB\to L^p(\mu)\|^p$.
\end{rem}

The proof of Theorem~\ref{boundedness} is the crux of the paper. We divide it into three steps.
\begin{enumerate}
    \item We first use the bounded oscillation of Bloch functions on each $Q_I^{\mathrm{up}}$ to reduce the embedding to estimates involving the values $f(z_I)$ and the masses $\mu(Q_I^{\mathrm{up}})$.

    \item We then represent every function in $\mathring{\calB}$ using the analytic polynomials $\Phi_I$, with coefficients controlled by its Bloch seminorm, and establish the converse estimate for bounded coefficients. This reduces the embedding to uniform estimates for finite linear combinations of the $\Phi_I$.

    \item Finally, we connect these estimates with the diagonal matrix bounds defining $\calB_{\calD,p}(\mu)$. The weights $h$ allow us to treat the general $L^p$ problem through the quadratic forms associated with $\Gamma_{\mu,h,N}$. We then use finite-dimensional duality to obtain the capacity characterization, treating $0<p<2$, $p=2$, and $p>2$ separately.
\end{enumerate}

Once the boundedness is settled, the compactness part follows from a standard modification of the boundedness argument. For $0<\rho<1$, write $S_\rho:=\{z\in\D:|z|>\rho\}$. We have the following compactness characterization.

\begin{thm}[Compactness]\label{thm:compact}
Let $0<p<\infty$, $\calD$ be any dyadic system on $\T$,
and $\mu$ be a finite positive Borel measure on $\D$.
Then \eqref{20261015eq01} is compact if and only if
\begin{equation}\label{20261021cond01}
\calB_{\calD,p}(\mu)<\infty,
\qquad \textrm{and} \qquad 
\lim_{\rho \to 1^{-}}\calB_{\calD,p}(\one_{S_\rho} \mu)=0.
\end{equation}
\end{thm}

In addition to characterizing boundedness and compactness, we
compare our capacity condition with the logarithmic moment
conditions \eqref{20261016cond01} and \eqref{20261016cond02} of Girela, Pel\'aez, P\'erez-Gonz\'alez and
R\"atty\"a \cite{GPPR}. We focus on the case $p=2$ and prove
in Section~\ref{Sec06} that
\begin{equation}\label{20261025eq40}
\int_{\D}\log\frac{e}{1-|z|}\,d\mu(z)
\lesssim \mu(\D)+\calB_{\calD,2}(\mu)
\lesssim
\int_{\D}\left(\log\frac{e}{1-|z|}\right)^2\,d\mu(z).
\end{equation}
Thus our capacity condition lies between the classical necessary
condition involving the first logarithmic moment and the
sufficient condition involving the square logarithmic moment.

\begin{rem}
The methods used to prove Theorems~\ref{boundedness} and \ref{thm:compact} also extend to the embedding $\operatorname{id}:\calQ_p\to L^q(\mu)$ for $0<p\le1$ and $0<q<\infty$; see Section~\ref{Sec07}. A further extension to $\calQ_K$ spaces is briefly discussed in Remark~\ref{20261025rem01}.
\end{rem}

The rest of the paper is organized as follows. In Section~\ref{Sec02}, we discretize the measure using the upper halves of dyadic Carleson tents. In Section~\ref{Sec03}, we establish the representation of functions in $\mathring{\calB}$ in terms of the polynomials $\Phi_I$ with bounded coefficients. We then prove the boundedness characterization, Theorem~\ref{boundedness}, in Section~\ref{Sec04}, and the compactness characterization, Theorem~\ref{thm:compact}, in Section~\ref{Sec05}. In Section~\ref{Sec06}, we compare our capacity with the classical logarithmic conditions, compute the capacity for a family of atomic measures, and construct two examples illustrating the distinction between these conditions, all in the case $p=2$. Finally, in Section~\ref{Sec07}, we extend the boundedness and compactness characterizations to $\calQ_p$ spaces.

Throughout the paper, $a\lesssim b$ means $a\le Cb$ for a positive constant $C$, and $a\gtrsim b$ means $b\lesssim a$. We write $a\simeq b$ when both inequalities hold. Unless stated otherwise, implicit constants may depend on the exponents and the fixed cutoff functions, but are independent of the measure and the choice of dyadic system. Constants denoted by $C$ may change from line to line.

\bigskip 

\section{Discretization of the measure $\mu$}\label{Sec02}

Fix a dyadic system $\calD$ on $\T$. For a finite positive
Borel measure $\mu$ on $\D$, define the atomic measure
\begin{equation}\label{20261016eq10}
\mu_{\calD}:=
\sum_{I\in\calD}\mu(Q_I^{\mathrm{up}})\delta_{z_I},
\end{equation}
where $\delta_{z_I}$ denotes Dirac measure at $z_I$.
By \eqref{eq:partition},
$$
\mu_{\calD}(\D)
=\sum_{I\in\calD}\mu(Q_I^{\mathrm{up}})
=\mu(\D).
$$
In this section, we show that $\mu$ and $\mu_{\calD}$ give
comparable Bloch embedding bounds for every $0<p<\infty$.
We begin with the following lemma.

\begin{lem}\label{20261016lem01}
For every $f\in\calB$, $I\in\calD$, and $z\in Q_I^{\mathrm{up}}$,
$$
|f(z)-f(z_I)|\lesssim\|f\|_{\calB,*},
$$
where $z_I$ is defined in \eqref{20261016defn01X}.
\end{lem}

\begin{proof}
This is a standard consequence of the Lipschitz characterization
of Bloch functions (see, e.g., \cite[Theorem~5.5]{Zhu2007}),
so we omit the proof.
\end{proof}

The following proposition is crucial. 

\begin{prop}\label{20261016prop01}
Let $0<p<\infty$,  $\mu$ be a finite positive Borel measure
on $\D$, and  $\mu_{\calD}$ be defined by
\eqref{20261016eq10}. For every $f\in\calB$, we have 
\begin{equation} \label{20261016eq11}
\int_{\D}|f|^p\,d\mu
\lesssim
\sum_{I\in\calD}\mu(Q_I^{\mathrm{up}})|f(z_I)|^p
+\mu(\D)\|f\|_{\calB,*}^p,
\end{equation} 
and
\begin{equation} \label{20261016eq12}
\sum_{I\in\calD}\mu(Q_I^{\mathrm{up}})|f(z_I)|^p
\lesssim
\int_{\D}|f|^p\,d\mu
+\mu(\D)\|f\|_{\calB,*}^p.
\end{equation}
Consequently,
\begin{equation}\label{20261016eq13}
\|\operatorname{id}:\calB\to L^p(\mu)\|
\simeq
\|\operatorname{id}:\calB\to L^p(\mu_{\calD})\|,
\end{equation}
with the implicit constants depending only on $p$.
\end{prop}

\begin{proof}
For any $p>0$, by Lemma \ref{20261016lem01}, we have 
$$
|f(z)|^p \lesssim |f(z_I)|^p+\|f\|_{\calB,*}^p, \qquad z \in Q_I^{\mathrm{up}}, 
$$
and 
$$
|f(z_I)|^p
 \lesssim |f(z)|^p+\|f\|_{\calB,*}^p,
\qquad z\in Q_I^{\mathrm{up}}.
$$
Then \eqref{20261016eq11} and \eqref{20261016eq12} follow by
integrating the above estimates over $Q_I^{\mathrm{up}}$, and then
summing over $I\in\calD$. 

To prove \eqref{20261016eq13}, recall that
$\mu_{\calD}(\D)=\mu(\D)$. Testing with the constant function
$1$, we find that 
$$
\mu(\D)^{1/p}
\le\|\operatorname{id}:\calB\to L^p(\mu)\|,
\qquad \textrm{and} \qquad 
\mu(\D)^{1/p}
\le\|\operatorname{id}:\calB\to L^p(\mu_{\calD})\|.
$$
Taking the supremum over $\|f\|_{\calB}\le1$ in
\eqref{20261016eq11}, we obtain
$$
\begin{aligned}
\|\operatorname{id}:\calB\to L^p(\mu)\|^p
&\lesssim
\|\operatorname{id}:\calB\to L^p(\mu_{\calD})\|^p+\mu(\D)\\
&\lesssim
\|\operatorname{id}:\calB\to L^p(\mu_{\calD})\|^p.
\end{aligned}
$$
The reverse inequality follows in the same way from
\eqref{20261016eq12}. This proves
\eqref{20261016eq13}.
\end{proof}

Consequently, the study of \eqref{20261015eq01x} reduces
to the following dyadic embedding problem:
\begin{equation}\label{20261015eq01xY}
\int_\D |f(z)|^p\,d\mu_{\calD}(z)
=\sum_{I\in\calD}\mu(Q_I^{\mathrm{up}})|f(z_I)|^p
\lesssim \|f\|_{\calB,*}^p,
\qquad f\in\mathring{\calB},
\end{equation}
where $\mu_{\calD}$ is defined in \eqref{20261016eq10}.

\medskip 

\section{Representation theory for $\mathring{\calB}$ functions}\label{Sec03}

The purpose of this section is to study the representation theory
of $\mathring{\calB}$ using the polynomials $\Phi_I$ defined in
\eqref{20261017eq05}. More precisely, we prove the following theorem.

\begin{thm}\label{20261017thm01}
\begin{enumerate}
    \item For every bounded sequence $(c_I)_{I\in\calD}$,
    the series
    $$
    f=\sum_{I\in\calD}c_I\Phi_I
    $$
    converges locally uniformly in $\D$ and defines a function
    $f\in\mathring{\calB}$ satisfying
    $$
    \|f\|_{\calB,*}\lesssim\sup_{I\in\calD}|c_I|.
    $$

    \item Conversely, every $f\in\mathring{\calB}$ admits a representation
    $$
    f=\sum_{I\in\calD}c_I\Phi_I,
    \qquad \text{with}\quad
    \sup_{I\in\calD}|c_I|\lesssim\|f\|_{\calB,*}.
    $$
\end{enumerate}
\end{thm}

\subsection{A preliminary estimate}

\begin{lem}\label{20261016lem20}
Let $q\in C^2(\mathbb R)$ be supported in $[1/16,3/4]$.
For an integer $M\ge1$, define
\begin{equation} \label{20261024defn01}
K_{M,q}(\zeta):=
\sum_{k\in\mathbb Z}q \left(\frac{k}{M} \right) \zeta^k,
\qquad \zeta\in\T.
\end{equation}
Then
\begin{equation}\label{20261017eq01}
|K_{M,q}(\zeta)|
\lesssim
\bigl(\|q\|_\infty+\|q''\|_\infty\bigr)
\frac{M}{(1+M|\zeta-1|)^2}.
\end{equation}
Consequently,
\begin{equation}\label{20261017eq02}
\|K_{M,q}\|_{L^1(\T)}
\lesssim \|q\|_\infty+\|q''\|_\infty.
\end{equation}
Moreover, for every $n\ge1$ and $\zeta\in\T$,
\begin{equation}\label{20261017eq03}
2^{-n}\sum_{I\in\calD_n}
\left|K_{2^n,q}(\zeta\overline{\zeta_I})\right|
\lesssim \|q\|_\infty+\|q''\|_\infty,
\end{equation}
where we recall $\zeta_I$ is the midpoint of the arc $I$. 
\end{lem}

\begin{proof}
We first show \eqref{20261017eq01}. Since the sum defining $K_{M,q}$ contains $\lesssim M$ nonzero
terms, we have
\begin{equation}\label{20261017eq04}
|K_{M,q}(\zeta)|\lesssim M\|q\|_\infty.
\end{equation}
On the other hand, 
\begin{align*}
\left|(1-\zeta)^2K_{M,q}(\zeta) \right|
& \le \sum_{k\in\mathbb Z}
\left| q\left( \frac{k}{M}  \right)-2q\left(\frac{k-1}{M} \right)+q\left(\frac{k-2}{M} \right)
\right| \left| \zeta \right| ^k \\
& \le \sum_{k\in\mathbb Z} \left| \int_0^{1/M}\int_0^{1/M}
q''\left(\frac{k}{M}-u-v\right)\,du\,dv \right| |\zeta|^k. 
\end{align*}
Since $q(k/M)-2q((k-1)/M)+q((k-2)/M)$ is nonzero for only
$\lesssim M$ values of $k$, the above computation therefore gives
$$
|(1-\zeta)^2K_{M,q}(\zeta)|
\lesssim M^{-1}\|q''\|_\infty.
$$
Combining this estimate with \eqref{20261017eq04} proves
\eqref{20261017eq01}. The estimate \eqref{20261017eq02} then follows
by integrating \eqref{20261017eq01} over $\T$.

\vspace{0.1cm}

Finally, to see \eqref{20261017eq03}, note that the points $\zeta_I$, $I\in\calD_n$, are equally spaced
on $\T$, and hence by \eqref{20261017eq01}, we obtain
\begin{align*}
2^{-n}\sum_{I\in\calD_n}
\left|K_{2^n,q}(\zeta\overline{\zeta_I})\right| 
& \lesssim \left(\|q\|_\infty+\|q''\|_\infty \right) \cdot \left(\sum_{I\in\calD_n}
\frac{1}{(1+2^n|\zeta-\zeta_I|)^2} \right) \\
&\lesssim \left(\|q\|_\infty+\|q''\|_\infty \right) \cdot \left(\sum_{j\ge0}\frac{1}{(1+j)^2} \right) \\
& \lesssim \|q\|_\infty+\|q''\|_\infty, 
\end{align*}
which gives the desired estimate. 
\end{proof}

\subsection{Proof of Theorem~\ref{20261017thm01}, (1)}

We need the following lemma. 

\begin{lem}\label{20261024lem01X}
There is a constant $c>0$ such that, for every $n\ge0$ and $z\in\D$, one has
\begin{equation} \label{20261017eq10}
\sum_{I\in\calD_n}|\Phi_I(z)| \lesssim e^{-c2^n(1-|z|)}
\end{equation}
and
\begin{equation} \label{20261017eq12X}
\sum_{I\in\calD_n}|\Phi_I'(z)| \lesssim 2^n e^{-c2^n(1-|z|)}, 
\end{equation}
where we recall $\Phi_I$ is the analytic polynomial defined as in \eqref{20261017eq05}. 
\end{lem}

\begin{proof}
Since $\Phi_{\T}=0$, we may assume that $n\ge1$. We first consider the case when $|z| \ge 1/2$. Write $z=r\zeta$ for some $1/2\le r<1$ and $\zeta\in\T$. Define
$$
q_{n,r}(s)=\eta(s)r^{2^ns}, \qquad \widetilde q_{n,r}(s)=s\eta(s)r^{2^ns}, 
$$
where $\eta$ is the smooth cut--off defined as in \eqref{20261017eq12}. Observe that both $q_{n,r}$ and $\widetilde q_{n,r}$ are supported in $[1/16,3/4]$, and satisfy
\begin{equation}\label{20261017eq13}
\begin{aligned}
\|q_{n,r}\|_\infty+\|q_{n,r}''\|_\infty
+\|\widetilde q_{n,r}\|_\infty+\|\widetilde q_{n,r}''\|_\infty
&\lesssim (1+2^n|\log r|)^2r^{2^n/16} \\
& \lesssim e^{-c2^n(1-r)}.
\end{aligned}
\end{equation}

By \eqref{20261017eq05}, for every $I\in\calD_n$,
$$
\Phi_I(r\zeta)=2^{-n}K_{2^n,q_{n,r}}(\zeta\overline{\zeta_I}), \qquad \textrm{and} \qquad 
\Phi_I'(r\zeta)= r^{-1}\overline{\zeta}\,K_{2^n,\widetilde q_{n,r}}(\zeta\overline{\zeta_I}).
$$
Since $r \ge 1/2$, by \eqref{20261017eq03} and \eqref{20261017eq13}, we have 
$$
\sum_{I\in\calD_n}|\Phi_I(r\zeta)|\lesssim e^{-c2^n(1-r)}, \qquad \textrm{and} \qquad 
\sum_{I\in\calD_n}|\Phi_I'(r\zeta)|\lesssim 2^n e^{-c2^n(1-r)}.
$$
This proves \eqref{20261017eq10} and \eqref{20261017eq12X} when $|z|\ge1/2$.

\vspace{0.1cm}

For $|z|<1/2$, observe that both $\sum_{I\in\calD_n}|\Phi_I(z)|$ and $\sum_{I\in\calD_n}|\Phi_I'(z)|$ are subharmonic. Therefore, by the maximum principle and the estimates on $|z|=1/2$, we see that 
$$
\sum_{I\in\calD_n}|\Phi_I(z)|\lesssim e^{-c2^{n-1}}, \qquad \textrm{and} \qquad 
\sum_{I\in\calD_n}|\Phi_I'(z)|\lesssim 2^n e^{-c2^{n-1}}.
$$
for $|z|<1/2$. The proof is complete. 
\end{proof}

\begin{proof}[Proof of Theorem~\ref{20261017thm01}, (1)]
Let $(c_I)_{I\in\calD}$ be a bounded complex sequence. By homogeneity, we may assume that $\sup_{I\in\calD}|c_I|\le1$.  Then by \eqref{20261017eq10}, we have for every $0<R<1$,
$$
\sup_{|z|\le R}\sum_{I\in\calD_n}|c_I\Phi_I(z)|\lesssim e^{-c2^n(1-R)}.
$$
Summing over $n$, we see that the series
$$
f(z):=\sum_{n\ge0}\sum_{I\in\calD_n}c_I\Phi_I(z)
$$
converges absolutely and uniformly on compact subsets of $\D$. In particular, $f$ is holomorphic and satisfies $f(0)=0$.

Similarly, the same argument as above with \eqref{20261017eq12X} gives
$$
f'(z)=\sum_{n \ge 0} \sum_{I\in\calD_n}c_I\Phi'_I(z)
$$
converges also absolutely and uniformly on compact subsets of $\D$. Moreover, one has 
$$
|f'(z)|\lesssim\sum_{n\ge0}2^n e^{-c2^n(1-|z|)}.
$$
Consequently, 
\begin{equation}\label{20261017eq14}
\begin{aligned}
|f'(z)|(1-|z|^2)
&\lesssim (1-|z|^2)\sum_{n\ge0}2^n e^{-c2^n(1-|z|)}\\
&\lesssim (1-|z|^2)\sum_{n\ge0}\int_{2^{n-1}}^{2^n}e^{-c(1-|z|)s}\,ds\\
&=(1-|z|^2)\int_{1/2}^{\infty}e^{-c(1-|z|)s}\,ds\\
&\lesssim\frac{1-|z|^2}{1-|z|}\lesssim1.
\end{aligned}
\end{equation}
Taking the supremum over $z\in\D$, we conclude that $f \in \mathrm{\calB}$. 
\end{proof}

\subsection{Proof of Theorem~\ref{20261017thm01}, (2)}

Fix a smooth nonincreasing function $\chi$ on $[0,\infty)$ such that
$$
0\le\chi\le1, \qquad \chi(s)=1, \; 0\le s\le1/4, \qquad \textrm{and} \qquad \chi(s)=0, s\ge 1/2.
$$
For $s>0$, set $\varphi(s)=\chi(s)-\chi(2s)$, and extend $\varphi$ by zero to $s\le0$. Then $\varphi$ is smooth, nonnegative, and supported in $[1/8,1/2]$, where $\eta=1$. Then for every integer $k\ge1$,
\begin{equation}\label{20261017eq15}
\sum_{n=0}^{N}\varphi(k2^{-n})=\chi(k2^{-N}), \qquad
\sum_{n\ge0}\varphi(k2^{-n})=1.
\end{equation}
For $f(z)=\sum_{k\ge1}a_kz^k\in\mathring{\calB}$, define
\begin{equation} \label{20261018eq02}
f_n(z):=\sum_{k\ge1}\varphi(k2^{-n})a_kz^k, \qquad n\ge0.
\end{equation} 
Observe each $f_n$ is an analytic polynomial whose nonzero coefficients have indices in $[2^{n-3},2^{n-1}]$. In particular, $f_0=f_1=0$.

\begin{lem}\label{20261017lem20}
For every $f\in\mathring{\calB}$ and $n\ge0$,
\begin{equation}\label{20261017eq40}
\sup_{|\zeta|=1}|f_n(\zeta)|\lesssim\|f\|_{\calB,*}.
\end{equation}
Moreover,
\begin{equation}\label{20261017eq41}
f=\sum_{n\ge0}f_n
\end{equation}
uniformly on compact subsets of $\D$.
\end{lem}

\begin{proof}
Since $f_0=f_1=0$, \eqref{20261017eq40} is immediate for $n=0,1$. For $n\ge2$, we first prove \eqref{20261017eq40}. Define
$$
q_n(s)=\frac{\varphi(s)}{s(1-2^{-n})^{2^ns}}, \qquad s>0,
$$
and extend $q_n$ by zero to $s\le0$. It is clear that $q_n$ is supported in $[1/8,1/2] \subset [1/16, 3/4]$. Moreover, a direct computation gives $\|q_n\|_\infty+\|q_n''\|_\infty\lesssim 1$, uniformly in $n\ge2$. Hence, by \eqref{20261017eq02} we see that 
\begin{equation}\label{20261017eq43}
\|K_{2^n,q_n}\|_{L^1(\T)}\lesssim 1, 
\end{equation}
where 
$$
K_{2^n,q_n}(\zeta)
=\sum_{k\ge1}q_n(k2^{-n})\zeta^k
=2^n\sum_{k\ge1}\frac{\varphi(k2^{-n})}{k(1-2^{-n})^k}\zeta^k,
\qquad \zeta\in\T.
$$
For any $\xi=e^{it}\in\T$, by Cauchy's integral formula followed by a change to polar coordinates, we have
\begin{align} \label{20261025eq03}
f_n(\xi)
&=\sum_{k\ge1}\varphi(k2^{-n})a_ke^{ikt}= \sum_{k \ge 1} \frac{\varphi(k2^{-n}) e^{ikt}}{k} \cdot ka_k \nonumber \\
&=\sum_{k\ge1}\frac{\varphi(k2^{-n})e^{ikt}}{2\pi i k}
\int_{|w|=1-2^{-n}}\frac{f'(w)}{w^k}\,dw \nonumber  \\
&=(1-2^{-n})\int_0^{2\pi}
\left(\sum_{k\ge1}\frac{\varphi(k2^{-n})e^{ik(t-s)}}{k(1-2^{-n})^k}\right)
e^{is}f'((1-2^{-n})e^{is})\,\frac{ds}{2\pi} \nonumber  \\
&=2^{-n}(1-2^{-n})\int_0^{2\pi}
K_{2^n,q_n}(e^{i(t-s)})e^{is}f'((1-2^{-n})e^{is})\,\frac{ds}{2\pi}.
\end{align}
Consequently, by \eqref{20261017eq43},
\begin{align*}
\sup_{|\xi|=1}|f_n(\xi)|
&\lesssim  2^{-n} \|K_{2^n,q_n}\|_{L^1(\T)}
\cdot \left(\sup_{|\xi|=1}|f'((1-2^{-n})\xi)| \right) \\
&\lesssim 2^{-n}\sup_{|\xi|=1}|f'((1-2^{-n})\xi)| \\
&\le \frac{2^{-n}}{1-(1-2^{-n})^2}\|f\|_{\calB,*}
\lesssim\|f\|_{\calB,*},
\end{align*}
which gives \eqref{20261017eq40}. 

\medskip

Finally, let $f(z)=\sum_{k\ge1}a_kz^k\in\mathring{\calB}$. Then by \eqref{20261017eq15}, we have 
$$
\sum_{n=0}^{N}f_n(z)=\sum_{k\ge1}\chi(k2^{-N})a_kz^k.
$$
For every $0<R<1$, we have $\sum_{k\ge1}|a_k|R^k<\infty$. Since $0\le\chi\le1$ and $\chi(k2^{-N})\to1$ for each $k$, an application of the dominated convergence theorem yields
$$
\sup_{|z|\le R}\left|f(z)-\sum_{n=0}^{N}f_n(z)\right|
\le\sum_{k\ge1}|1-\chi(k2^{-N})|\,|a_k|R^k\longrightarrow0.
$$
This proves \eqref{20261017eq41}.
\end{proof}

\begin{proof}[Proof of Theorem~\ref{20261017thm01}, (2)]
Let $f\in\mathring{\calB}$ and choose
$$
c_{\T}=0, \qquad c_I=f_n(\zeta_I) \quad \textrm{for} \quad I\in\calD_n,\ n \ge 1.
$$
By Lemma~\ref{20261017lem20},
$$
\sup_{I\in\calD}|c_I|\lesssim\|f\|_{\calB,*}.
$$
We claim that, for every $n\ge1$,
\begin{equation}\label{20261018eq03}
f_n(z)=\sum_{I\in\calD_n}f_n(\zeta_I)\Phi_I(z).
\end{equation}
The case $n=1$ is immediate because $f_1=0$. For $n\ge2$, note that the points $\zeta_I$, $I\in\calD_n$, form a rotation of the $2^n$th roots of unity. Hence,
\begin{equation}\label{20261018eq01}
2^{-n}\sum_{I\in\calD_n}\zeta_I^\ell\overline{\zeta_I}^{\,k}
=\begin{cases}
1,&k=\ell,\\
0,&k\ne\ell,
\end{cases}
\qquad 1\le k,\ell<2^n.
\end{equation}
By \eqref{20261018eq02} and \eqref{20261017eq05}, respectively, we have, for $n\ge1$ and $I\in\calD_n$, 
$$
f_n(z)=\sum_{\ell\ge1}\varphi(\ell2^{-n})a_\ell z^\ell,
\qquad \textrm{and} \qquad 
\Phi_I(z)=2^{-n}\sum_{k\ge1}\eta(k2^{-n})\overline{\zeta_I}^{\,k}z^k.
$$
By the support conditions on $\varphi$ and $\eta$, every nonzero coefficient in the above expressions satisfies $1\le k,\ell<2^n$. Therefore, by \eqref{20261018eq01}
\begin{align*}
\textrm{RHS of \eqref{20261018eq03}}
&=2^{-n}\sum_{I\in\calD_n}
\left(\sum_{\ell\ge1}\varphi(\ell2^{-n})a_\ell\zeta_I^\ell\right)
\left(\sum_{k\ge1}\eta(k2^{-n})\overline{\zeta_I}^{\,k}z^k\right) \nonumber \\
&=2^{-n}\sum_{I\in\calD_n}\sum_{\ell,k\ge1}
\varphi(\ell2^{-n})a_\ell\eta(k2^{-n})z^k\zeta_I^\ell\overline{\zeta_I}^{\,k} \nonumber \\
&=\sum_{\ell,k\ge1}\varphi(\ell2^{-n})a_\ell\eta(k2^{-n})z^k
\left(2^{-n}\sum_{I\in\calD_n}\zeta_I^\ell\overline{\zeta_I}^{\,k}\right) \nonumber \\
&= \sum_{k \ge 1} \varphi(k2^{-n}) a_k \eta(k2^{-n}) z^k \nonumber \\
&=\sum_{k\ge1}\varphi(k2^{-n})a_kz^k=f_n(z), \nonumber 
\end{align*}
where in the last equation, we have used $\eta=1$ on the support of $\varphi$. This proves \eqref{20261018eq03}.

\vspace{0.1cm}

To this end, summing \eqref{20261018eq03} over $n$ and using $f_0=0$, $\Phi_{\T}=0$, and \eqref{20261017eq41}, we obtain
$$
f(z)=\sum_{n\ge0}\sum_{I\in\calD_n}c_I\Phi_I(z).
$$
The proof is complete. 
\end{proof}

\bigskip

\section{Proof of the boundedness theorem}\label{Sec04}

The purpose of this section is to prove Theorem~\ref{boundedness}. Recall that by Proposition~\ref{20261016prop01}, our main goal is to study the dyadic embedding problem \eqref{20261015eq01xY}: 
$$
\sum_{I\in\calD}\mu(Q_I^{\mathrm{up}})|f(z_I)|^p
\lesssim \|f\|_{\calB,*}^p,
\qquad f\in\mathring{\calB},
$$
where $\mu_{\calD}$ is defined in \eqref{20261016eq10}. Also recall that $\mu(\D)<+\infty$. Moreover, we may assume that $\mu(\D)>0$, since the case $\mu=0$ is trivial. 

\subsection{Reduction to a dyadic testing condition}

Define
\begin{equation}\label{20261018eq05}
C_{\calD,p}(\mu):=
\sup_{N\ge0}\ \sup_{\substack{|c_I|\le1\\I\in\calD_{\le N}}}
\left(\sum_{K\in\calD}\mu(Q_K^{\mathrm{up}})
\left|\sum_{I\in\calD_{\le N}}c_I\Phi_I(z_K)\right|^p
\right)^{2/p}.
\end{equation}

We first show that the boundedness of the dyadic embedding \eqref{20261015eq01xY} is equivalent to the finiteness of the dyadic testing condition $C_{\calD, p}(\mu)<+\infty$.
 
\begin{prop}\label{20261018prop01}
Let $0<p<\infty$ and let $\mu$ be a finite positive Borel measure on $\D$. Then
\begin{equation}\label{20261018eq20}
C_{\calD,p}(\mu)\simeq
\sup_{\substack{f\in\mathring{\calB}\\\|f\|_{\calB,*}\le1}}
\left(
\sum_{K\in\calD}\mu(Q_K^{\mathrm{up}})|f(z_K)|^p
\right)^{2/p}.
\end{equation}
Consequently,
\begin{equation}\label{20261018eq21}
\|\operatorname{id}:\calB\to L^p(\mu_{\calD})\|
\simeq \mu(\D)^{1/p}+C^{1/2}_{\calD,p}(\mu).
\end{equation}
\end{prop}

\begin{proof}
We first prove
\begin{equation} \label{20261018eq30}
C_{\calD,p}(\mu) \lesssim 
\sup_{\substack{f\in\mathring{\calB}\\\|f\|_{\calB,*}\le1}}
\left(
\sum_{K\in\calD}\mu(Q_K^{\mathrm{up}})|f(z_K)|^p
\right)^{2/p}.
\end{equation} 
Fix $N\ge0$ and coefficients $|c_I|\le1$, $I\in\calD_{\le N}$. By Theorem~\ref{20261017thm01}, (1), 
$$
f=\sum_{I\in\calD_{\le N}}c_I\Phi_I \in \mathring{\calB}
$$
and satisfies $\|f\|_{\calB,*}\lesssim1$. Hence
$$
\left(
\sum_{K\in\calD}\mu(Q_K^{\mathrm{up}})
\left|\sum_{I\in\calD_{\le N}}c_I\Phi_I(z_K)\right|^p
\right)^{2/p}
\lesssim
\sup_{\substack{f\in\mathring{\calB}\\\|f\|_{\calB,*}\le1}}
\left(
\sum_{K\in\calD}\mu(Q_K^{\mathrm{up}})|f(z_K)|^p \right)^{2/p}.
$$
This gives \eqref{20261018eq30} by 
taking the supremum over the coefficients and then over $N$.

\vspace{0.1cm}

For the reverse inequality, we may assume that $C_{\calD,p}(\mu)<\infty$. Let $f\in\mathring{\calB}$ with $\|f\|_{\calB,*}\le1$. By Theorem~\ref{20261017thm01}, (2), there are coefficients $(c_I)_{I\in\calD}$ such that
$$
f=\sum_{n\ge0}\sum_{I\in\calD_n}c_I\Phi_I,
\qquad
\sup_{I\in\calD}|c_I|\lesssim1,
$$
where the series converges locally uniformly in $\D$. By Fatou's lemma and \eqref{20261018eq05},
\begin{align*}
\sum_{K\in\calD}\mu(Q_K^{\mathrm{up}})|f(z_K)|^p
&\le \liminf_{N\to\infty}
\sum_{K\in\calD}\mu(Q_K^{\mathrm{up}})
\left|\sum_{I\in\calD_{\le N}}c_I\Phi_I(z_K)\right|^p\\
&\lesssim C_{\calD,p}(\mu)^{p/2},
\end{align*}
which clearly gives \eqref{20261018eq20}.

\vspace{0.1cm}

Finally, for an arbitrary $f\in\calB$, applying \eqref{20261018eq20} to $f-f(0)$, we obtain
$$
\int_{\D}|f|^p\,d\mu_{\calD}
\lesssim
\mu(\D)|f(0)|^p
+C_{\calD,p}(\mu)^{p/2}\|f\|_{\calB,*}^p.
$$
This gives $\|\operatorname{id}:\calB\to L^p(\mu_{\calD})\| \lesssim \mu(\D)^{1/p}+C^{1/2}_{\calD,p}(\mu)$.
The reverse direction follows by testing with the constant function $1$ and \eqref{20261018eq20}.
\end{proof}

By Proposition~\ref{20261016prop01} and Proposition~\ref{20261018prop01}, Theorem~\ref{boundedness} will follow once we prove
\begin{equation}\label{20261020mainest01}
C_{\calD,p}(\mu)\le \calB_{\calD,p}(\mu)\lesssim C_{\calD,p}(\mu).
\end{equation}
We divide the proof into two parts. 

\subsection{Proof of $C_{\calD,p}(\mu)\le\calB_{\calD,p}(\mu)$}

Assume that $\calB_{\calD,p}(\mu)<\infty$. We shall prove the first inequality in \eqref{20261020mainest01}. Fix $N\ge0$ and coefficients $|c_I|\le1$, $I\in\calD_{\le N}$. Note that if a nonnegative sequence $h=(h_K)_{K\in\calD}$ and numbers $\lambda_I\ge0$ satisfy
$$
\Gamma_{\mu,h,N}\le\operatorname{diag}(\lambda_I),
$$
the identity \eqref{eq:gram-form} gives
\begin{equation}\label{20261018eq22}
\sum_{K\in\calD}\mu(Q_K^{\mathrm{up}})h_K
\left|\sum_{I\in\calD_{\le N}}c_I\Phi_I(z_K)\right|^2
=c^*\Gamma_{\mu,h,N}c\le\sum_{I\in\calD_{\le N}}\lambda_I|c_I|^2
\le\sum_{I\in\calD_{\le N}}\lambda_I.
\end{equation}
Now we consider three different cases.

\subsubsection{The case $0<p<2$}
Recall that in this case, 
$$
\calB_{\calD,p}(\mu):=
\sup_{N\ge0}\ \inf_{\substack{h_K>0\\
\sum_{K\in\calD}\mu(Q_K^{\mathrm{up}})h_K^{\frac{p}{p-2}}\le1}}
\inf\left\{ \sum_{I\in\calD_{\le N}}\lambda_I: \lambda_I\ge0,\quad
\Gamma_{\mu,h,N}\le\operatorname{diag}(\lambda_I) \right\}.
$$
Let $h=(h_K)_{K \in \calD}$ be a positive sequence satisfying 
$$
\sum_{K\in\calD}\mu(Q_K^{\mathrm{up}})h_K^{p/(p-2)}\le1.
$$
For any $\lambda_I\ge0$ satisfying $\Gamma_{\mu,h,N}\le\operatorname{diag}(\lambda_I)$, by H\"older and \eqref{20261018eq22}, one has
\begin{align*}
&\left(
\sum_{K\in\calD}\mu(Q_K^{\mathrm{up}})
\left|\sum_{I\in\calD_{\le N}}c_I\Phi_I(z_K)\right|^p
\right)^{2/p}\\
&\quad\le
\left(
\sum_{K\in\calD}\mu(Q_K^{\mathrm{up}})h_K
\left|\sum_{I\in\calD_{\le N}}c_I\Phi_I(z_K)\right|^2
\right)
\left(
\sum_{K\in\calD}\mu(Q_K^{\mathrm{up}})h_K^{p/(p-2)}
\right)^{(2-p)/p}\\
&\quad\le\sum_{I\in\calD_{\le N}}\lambda_I.
\end{align*}
Taking the infimum first over $\lambda$ and then over $h$, we obtain
$$
\left(
\sum_{K\in\calD}\mu(Q_K^{\mathrm{up}})
\left|\sum_{I\in\calD_{\le N}}c_I\Phi_I(z_K)\right|^p
\right)^{2/p}
\le\calB_{\calD,p}(\mu), 
$$
which further implies $C_{\calD,p}(\mu)\le\calB_{\calD,p}(\mu)$. 

\subsubsection{The case $p=2$} 
Recall that in this case
$$
\calB_{\calD,2}(\mu):=
\sup_{N\ge0}
\inf\left\{
\sum_{I\in\calD_{\le N}}\lambda_I:
\lambda_I\ge0,\quad
\Gamma_{\mu,N}\le\operatorname{diag}(\lambda_I)
\right\}.
$$
Taking $h_K=1$ in \eqref{20261018eq22}, we have
$$
\sum_{K\in\calD}\mu(Q_K^{\mathrm{up}})
\left|\sum_{I\in\calD_{\le N}}c_I\Phi_I(z_K)\right|^2
\le
\inf\left\{
\sum_{I\in\calD_{\le N}}\lambda_I:
\lambda_I\ge0,\quad
\Gamma_{\mu,N}\le\operatorname{diag}(\lambda_I)
\right\}.
$$
The desired estimate $C_{\calD,2}(\mu)\le\calB_{\calD,2}(\mu)$ then follows by taking the supremum over the coefficients and $N$.

\subsubsection{The case $p>2$}
Recall that in this case, one has
$$
\calB_{\calD,p}(\mu):=
\sup_{N\ge0}\ 
\sup_{\substack{h_K\ge0\\
\sum_{K\in\calD}\mu(Q_K^{\mathrm{up}})h_K^{\frac{p}{p-2}}\le1}}
\inf\left\{
\sum_{I\in\calD_{\le N}}\lambda_I:
\lambda_I\ge0,\quad
\Gamma_{\mu,h,N}\le\operatorname{diag}(\lambda_I)
\right\}.
$$
First, note that by \eqref{eq:individual-bound}, the sum
$$
\sum_{K\in\calD}\mu(Q_K^{\mathrm{up}})
\left|\sum_{I\in\calD_{\le N}}c_I\Phi_I(z_K)\right|^p \lesssim \left(\# \calD_{\le N} \right)^p \cdot \mu(\D)<+\infty.
$$
If it vanishes, the desired estimate is clear. Otherwise, choose
$$
h_K=
\frac{
\left|\sum_{I\in\calD_{\le N}}c_I\Phi_I(z_K)\right|^{p-2}
}{
\left(
\sum_{L\in\calD}\mu(Q_L^{\mathrm{up}})
\left|\sum_{I\in\calD_{\le N}}c_I\Phi_I(z_L)\right|^p
\right)^{(p-2)/p}
}.
$$
A direct computation gives
$$
\sum_{K\in\calD}\mu(Q_K^{\mathrm{up}})h_K^{p/(p-2)}=1
$$
and
$$
\sum_{K\in\calD}\mu(Q_K^{\mathrm{up}})h_K
\left|\sum_{I\in\calD_{\le N}}c_I\Phi_I(z_K)\right|^2
=\left(
\sum_{K\in\calD}\mu(Q_K^{\mathrm{up}})
\left|\sum_{I\in\calD_{\le N}}c_I\Phi_I(z_K)\right|^p
\right)^{2/p}.
$$
Thus, by \eqref{20261018eq22}, we have
\begin{align*}
&\left(
\sum_{K\in\calD}\mu(Q_K^{\mathrm{up}})
\left|\sum_{I\in\calD_{\le N}}c_I\Phi_I(z_K)\right|^p
\right)^{2/p}\\
&\quad\le
\inf\left\{
\sum_{I\in\calD_{\le N}}\lambda_I:
\lambda_I\ge0,\quad
\Gamma_{\mu,h,N}\le\operatorname{diag}(\lambda_I)
\right\}
\le\calB_{\calD,p}(\mu).
\end{align*}
Taking the supremum over the coefficients and $N$ proves $C_{\calD,p}(\mu)\le\calB_{\calD,p}(\mu)$.

\subsection{Proof of $\calB_{\calD,p}(\mu)\lesssim C_{\calD,p}(\mu)$}

We now prove the second inequality in \eqref{20261020mainest01}. We may assume that
$C_{\calD,p}(\mu)<\infty$. 
We first establish the following estimate. 

\begin{lem}\label{20261019lem01}
Let $0<p<\infty$ and $N\ge0$. For every finite family of coefficient vectors $c^{(1)}:=(c^{(1)}_I)_{I \in \calD_{\le N}}, \ldots, c^{(s)}:=(c^{(s)}_I)_{I \in \calD_{\le N}}$, we have
\begin{equation}\label{20261019eq01}
\left[\sum_{K\in\calD}\mu(Q_K^{\mathrm{up}})
\left(\sum_{j=1}^{s}\left|\sum_{I\in\calD_{\le N}}c_I^{(j)}\Phi_I(z_K)\right|^2\right)^{p/2}\right]^{2/p} \lesssim C_{\calD,p}(\mu) \cdot \max_{I\in\calD_{\le N}}\sum_{j=1}^{s}|c_I^{(j)}|^2.
\end{equation}
\end{lem}

\begin{proof}
Consider the linear map
$$
(c_I)_{I\in\calD_{\le N}}\longmapsto\sum_{I\in\calD_{\le N}}c_I\Phi_I
$$
from $\ell^\infty(\calD_{\le N})$ into $L^p(\mu_{\calD})$. By \eqref{20261018eq05}, its operator norm is at most $C_{\calD,p}(\mu)^{1/2}$. Now we consider three different cases. 

\vspace{0.1cm}

\noindent \textit{\underline{Case I}:} For $p\ge1$, the Grothendieck--Krivine square-function inequality gives the required estimate for real linear operators from $\ell^\infty(\calD_{\le N})$ into $L^p(\mu_{\calD})$; see \cite[Corollary~2.6]{Pisier}.

\vspace{0.1cm}

\noindent \textit{\underline{Case II}:} For $0<p<1$, the corresponding estimate follows from \cite[Corollary~3.2]{Kalton}. To apply this result, we may regard $\mu_{\calD}/\mu(\D)$ as a probability measure on $\overline{\D}$.

\vspace{0.1cm}

\noindent \textit{\underline{Case III}:} Finally, the complex case follows by decomposing each coefficient vector $c^{(j)}$, $1\le j\le s$, and each polynomial $\Phi_I$, $I\in\calD_{\le N}$, into their real and imaginary parts and applying the preceding estimates to the resulting real linear maps.
\end{proof}

Fix $N\ge0$. To prove $\calB_{\calD,p}(\mu)\lesssim C_{\calD,p}(\mu)$, we first consider the infimum
$$
\inf\left\{\sum_{I\in\calD_{\le N}}\lambda_I:
\lambda_I\ge0,\quad \Gamma_{\mu,h,N}\le\operatorname{diag}(\lambda_I)\right\},
$$
where we recall that $h=(h_K)_{K\in\calD}$ is a nonnegative sequence. We shall express this infimum in a form to which Lemma~\ref{20261019lem01} can be applied.

\begin{lem}\label{20261019lem02}
Let $N\ge0$ and let $h=(h_K)_{K\in\calD}$ be a nonnegative sequence. For a positive semidefinite Hermitian matrix $X$ indexed by $\calD_{\le N}$, define
\begin{equation}\label{20261019eq06}
a_K(X):=\sum_{I,J\in\calD_{\le N}}X_{JI}\overline{\Phi_I(z_K)}\Phi_J(z_K), \qquad K\in\calD.
\end{equation}
Then $a_K(X)\ge0$, and
\begin{equation}\label{20261019eq05}
\inf\left\{\sum_{I\in\calD_{\le N}}\lambda_I:
\lambda_I\ge0,\quad \Gamma_{\mu,h,N}\le\operatorname{diag}(\lambda_I)\right\}=\sup_{\substack{X\ge0\\X_{II}\le1}}
\sum_{K\in\calD}\mu(Q_K^{\mathrm{up}})h_Ka_K(X).
\end{equation}
\end{lem}

\begin{proof}
The claim $a_K(X) \ge 0$ follows from the positivity of $X$. To prove \eqref{20261019eq05}, we follow a similar finite-dimensional duality argument in \cite[Section~4.2]{HZBloch}.

First suppose that all entries of $\Gamma_{\mu,h,N}$ are finite. We associate a positive semidefinite matrix $X$ with the constraint
$$
\operatorname{diag}(\lambda_I)-\Gamma_{\mu,h,N}\ge0.
$$
Using the trace pairing, the corresponding Lagrangian is
$$
\sum_{I\in\calD_{\le N}}\lambda_I-
\operatorname{Tr}\left[X\left(\operatorname{diag}(\lambda_I)-\Gamma_{\mu,h,N}\right)\right]=\operatorname{Tr}(\Gamma_{\mu,h,N}X)
+\sum_{I\in\calD_{\le N}}\lambda_I(1-X_{II}).
$$
Its infimum over $\lambda_I\ge0$ is $\operatorname{Tr}(\Gamma_{\mu,h,N}X)$ if $X_{II}\le1$ for every $I$, and is $-\infty$ otherwise. Moreover, choosing all $\lambda_I$ equal to a sufficiently large positive constant makes $\operatorname{diag}(\lambda_I)-\Gamma_{\mu,h,N}$ positive definite, so Slater's condition is satisfied. Therefore, by finite-dimensional semidefinite duality, we have 
$$
\begin{aligned}
&\inf\left\{\sum_{I\in\calD_{\le N}}\lambda_I:
\lambda_I\ge0,\quad \Gamma_{\mu,h,N}\le\operatorname{diag}(\lambda_I)\right\}=\sup_{\substack{X\ge0\\X_{II}\le1}}\operatorname{Tr}(\Gamma_{\mu,h,N}X);
\end{aligned}
$$
see \cite[Section~3]{VB}. By the definitions of the trace, $\Gamma_{\mu,h,N}$, and $a_K(X)$, we have
\begin{align*}
\operatorname{Tr}(\Gamma_{\mu,h,N}X)
&=\sum_{I\in\calD_{\le N}}(\Gamma_{\mu,h,N}X)_{II}=\sum_{I,J\in\calD_{\le N}}\Gamma_{\mu,h,N}(I,J)X_{JI}\\
&=\sum_{I,J\in\calD_{\le N}}
\left(\sum_{K\in\calD}\mu(Q_K^{\mathrm{up}})h_K
\overline{\Phi_I(z_K)}\Phi_J(z_K)\right)X_{JI}\\
&=\sum_{K\in\calD}\mu(Q_K^{\mathrm{up}})h_K
\left(\sum_{I,J\in\calD_{\le N}}X_{JI}\overline{\Phi_I(z_K)}\Phi_J(z_K)\right)\\
&=\sum_{K\in\calD}\mu(Q_K^{\mathrm{up}})h_Ka_K(X).
\end{align*}
This proves \eqref{20261019eq05} when those entries are finite.

\vspace{0.1cm}

If $\Gamma_{\mu,h,N}(I_0,I_0)=+\infty$ for some $I_0\in\calD_{\le N}$, we take
$$
X_{IJ}=\begin{cases}
1,&I=J=I_0,\\
0,&\text{otherwise}.
\end{cases}
$$
Then $X\ge0$, $X_{II}\le1$, and $a_K(X)=|\Phi_{I_0}(z_K)|^2$. Hence
$$
\sum_{K\in\calD}\mu(Q_K^{\mathrm{up}})h_Ka_K(X)
=\sum_{K\in\calD}\mu(Q_K^{\mathrm{up}})h_K|\Phi_{I_0}(z_K)|^2
=\Gamma_{\mu,h,N}(I_0,I_0)=+\infty.
$$
Hence, both sides of \eqref{20261019eq05} are $+\infty$, hence the identity also holds in this case.
\end{proof}

We next apply Lemma~\ref{20261019lem01} to the quantities $a_K(X)$ in \eqref{20261019eq05}.

\begin{lem}\label{20261019lem03}
Let $0<p<\infty$ and $N\ge0$. Then
\begin{equation}\label{20261019eq07}
\sup_{\substack{X\ge0\\X_{II}\le1}}
\left(\sum_{K\in\calD}\mu(Q_K^{\mathrm{up}})a_K(X)^{p/2}\right)^{2/p}
\lesssim C_{\calD,p}(\mu),
\end{equation}
where $X$ is indexed by $\calD_{\le N}$ and $a_K(X)$ is defined in \eqref{20261019eq06}. 
\end{lem}

\begin{proof}
Let $X\ge0$ with $X_{II}\le1$ for every $I\in\calD_{\le N}$. Since $X$ is positive semidefinite, there are finitely many vectors $c^{(1)},\ldots,c^{(s)}$ such that
$$
X=\sum_{j=1}^{s}c^{(j)}(c^{(j)})^*.
$$
Consequently,
$$
\sum_{j=1}^{s}|c_I^{(j)}|^2=X_{II}\le1, \qquad I\in\calD_{\le N},
$$
and, by \eqref{20261019eq06},
\begin{align*}
a_K(X)=\sum_{I,J\in\calD_{\le N}} \left(\sum_{j=1}^{s}
c_J^{(j)}\overline{c_I^{(j)}} \right)
\overline{\Phi_I(z_K)}\Phi_J(z_K)=\sum_{j=1}^{s}\left|\sum_{I\in\calD_{\le N}}c_I^{(j)}\Phi_I(z_K)\right|^2.
\end{align*}
Therefore, by Lemma~\ref{20261019lem01}, we have
$$
\left(\sum_{K\in\calD}\mu(Q_K^{\mathrm{up}})a_K(X)^{p/2}\right)^{2/p}
\lesssim C_{\calD,p}(\mu)\max_{I\in\calD_{\le N}}\sum_{j=1}^{s}|c_I^{(j)}|^2
\le C_{\calD,p}(\mu), 
$$
which gives the desired estimate \eqref{20261019eq07}.
\end{proof}

We now combine Lemmas~\ref{20261019lem02} and \ref{20261019lem03} to prove $\calB_{\calD,p}(\mu)\lesssim C_{\calD,p}(\mu)$, considering the cases $0<p<2$, $p=2$, and $p>2$ separately.

\subsubsection{The case $0<p<2$}

In this case, we claim that, for each fixed $N$,
\begin{equation}\label{20261020eq30}
\begin{aligned}
&\inf_{\substack{h_K>0\\
\sum_{K\in\calD}\mu(Q_K^{\mathrm{up}})h_K^{p/(p-2)}\le1}}
\inf\left\{\sum_{I\in\calD_{\le N}}\lambda_I:
\lambda_I\ge0,\quad \Gamma_{\mu,h,N}\le\operatorname{diag}(\lambda_I)\right\}\\
&\qquad \qquad \le\sup_{\substack{X\ge0\\X_{II}\le1}}
\left(\sum_{K\in\calD}\mu(Q_K^{\mathrm{up}})a_K(X)^{p/2}\right)^{2/p}.
\end{aligned}
\end{equation}
Together with \eqref{20261019eq07}, this will give the desired estimate $\calB_{\calD,p}(\mu)\lesssim C_{\calD,p}(\mu)$ for the case $0<p<2$. 

We now prove the above claim. By \eqref{20261019eq05}, we have
\begin{align} \label{20261020eq01}
&\inf_{\substack{h_K>0\\
\sum_{K\in\calD}\mu(Q_K^{\mathrm{up}})h_K^{p/(p-2)}\le1}}
\inf\left\{\sum_{I\in\calD_{\le N}}\lambda_I:
\lambda_I\ge0,\quad \Gamma_{\mu,h,N}\le\operatorname{diag}(\lambda_I)\right\} \nonumber \\
&\qquad=\inf_{\substack{h_K>0\\
\sum_{K\in\calD}\mu(Q_K^{\mathrm{up}})h_K^{p/(p-2)}\le1}}
\left[ \sup_{\substack{X\ge0\\X_{II}\le1}}
\sum_{K\in\calD}\mu(Q_K^{\mathrm{up}})h_Ka_K(X) \right].
\end{align}
Our goal is now to choose a positive sequence $h=(h_K)_{K\in\calD}$ satisfying the constraint in \eqref{20261020eq01} so as to make the corresponding supremum as small as possible.

For this purpose, let
$$
\mathcal M_N:=\left\{
X=(X_{IJ})_{I,J\in\calD_{\le N}}:
X=X^*,\quad X\ge0,\quad X_{II}\le1\ \text{for every }I\in\calD_{\le N}
\right\}.
$$
For each fixed $\varepsilon>0$, we consider the maximization problem
\begin{equation}\tag{$P_\varepsilon$}\label{20261020eq02X}
\left\{
\begin{aligned}
&X\in\mathcal M_N,\\
&\displaystyle\max_X\sum_{K\in\calD}\mu(Q_K^{\mathrm{up}})
\bigl(a_K(X)+\varepsilon\bigr)^{p/2}.
\end{aligned}
\right.
\end{equation}

We first show that \eqref{20261020eq02X} admits a solution. It suffices to show that $\mathcal M_N$ is nonempty and compact. The first assertion is clear, since $\mathcal M_N$ contains the zero matrix. To prove compactness, we first observe that $\mathcal M_N$ is closed. Moreover, by positivity, we have 
$$
|X_{IJ}|^2\le X_{II}X_{JJ}\le1,
\qquad X\in\mathcal M_N,\quad I,J\in\calD_{\le N},
$$
which shows that $\mathcal M_N$ is bounded and hence compact. We denote a maximizer by $X_\varepsilon\in\mathcal M_N$.

\vspace{0.1cm}

Using this maximizer, we now define
\begin{equation}\label{20261020eq02}
h_K:=\left(\sum_{L\in\calD}\mu(Q_L^{\mathrm{up}}) \bigl(a_L(X_\varepsilon)+\varepsilon\bigr)^{p/2}\right)^{(2-p)/p}
\bigl(a_K(X_\varepsilon)+\varepsilon\bigr)^{(p-2)/2}.
\end{equation}
We first verify that this sequence satisfies the constraint in \eqref{20261020eq01}. Since $\mu(\D)>0$ and $\varepsilon>0$, we have $h_K>0$. Moreover,
$$
\sum_{K\in\calD}\mu(Q_K^{\mathrm{up}})h_K^{p/(p-2)}
=\frac{\sum_{K\in\calD}\mu(Q_K^{\mathrm{up}})
\bigl(a_K(X_\varepsilon)+\varepsilon\bigr)^{p/2}}
{\sum_{L\in\calD}\mu(Q_L^{\mathrm{up}})
\bigl(a_L(X_\varepsilon)+\varepsilon\bigr)^{p/2}}
=1.
$$
Thus $h$ satisfies the required constraint in \eqref{20261020eq01}. 

\vspace{0.1cm}

We next estimate the supremum in \eqref{20261020eq01} for this choice of $h$ by showing that $X_\varepsilon$ also maximizes the corresponding weighted sum
$$
\sum_{K\in\calD}\mu(Q_K^{\mathrm{up}})h_Ka_K(X)
$$
over $\mathcal M_N$. To see this, fix $X\in\mathcal M_N$. Since $\mathcal M_N$ is convex, we have 
$$
X_\varepsilon+t(X-X_\varepsilon)\in\mathcal M_N,
\qquad 0\le t\le1.
$$
Since $X_\varepsilon$ is a maximizer of the problem \eqref{20261020eq02X}, the function
$$
F_\varepsilon(t):=\sum_{K\in\calD}\mu(Q_K^{\mathrm{up}})
\left(a_K\bigl(X_\varepsilon+t(X-X_\varepsilon)\bigr)+\varepsilon\right)^{p/2},
\qquad 0\le t\le1,
$$
attains its maximum at $t=0$. Hence, the right derivative of $F_\varepsilon$ at $t=0$ is nonpositive, that is, 
\begin{equation}\label{20261020eq15}
\sum_{K\in\calD}\mu(Q_K^{\mathrm{up}})
\bigl(a_K(X_\varepsilon)+\varepsilon\bigr)^{(p-2)/2}
\bigl(a_K(X)-a_K(X_\varepsilon)\bigr)\le0.
\end{equation}
Multiplying both sides of the above estimate by
$$
\left(\sum_{L\in\calD}\mu(Q_L^{\mathrm{up}})
\bigl(a_L(X_\varepsilon)+\varepsilon\bigr)^{p/2}\right)^{(2-p)/p},
$$
and then by \eqref{20261020eq02}, we obtain
$$
\sum_{K\in\calD}\mu(Q_K^{\mathrm{up}})h_Ka_K(X)
\le\sum_{K\in\calD}\mu(Q_K^{\mathrm{up}})h_Ka_K(X_\varepsilon),
\qquad X\in\mathcal M_N.
$$
which gives the desired claim. Consequently,
\begin{align} \label{20261020eq20}
\sup_{X\in\mathcal M_N}
\sum_{K\in\calD}\mu(Q_K^{\mathrm{up}})h_Ka_K(X)
&=\sum_{K\in\calD}\mu(Q_K^{\mathrm{up}})h_Ka_K(X_\varepsilon) \nonumber \\
&\le\sum_{K\in\calD}\mu(Q_K^{\mathrm{up}})h_K
\bigl(a_K(X_\varepsilon)+\varepsilon\bigr) \nonumber \\
&=\left(\sum_{K\in\calD}\mu(Q_K^{\mathrm{up}}) 
\bigl(a_K(X_\varepsilon)+\varepsilon\bigr)^{p/2}\right)^{2/p}.
\end{align}
Since $0<p/2<1$, we have 
\begin{align} \label{20261020eq21}
\sum_{K\in\calD}\mu(Q_K^{\mathrm{up}})
\bigl(a_K(X_\varepsilon)+\varepsilon\bigr)^{p/2}
&\le\sum_{K\in\calD}\mu(Q_K^{\mathrm{up}})a_K(X_\varepsilon)^{p/2}
+\mu(\D)\varepsilon^{p/2} \nonumber \\
&\le\sup_{X\in\mathcal M_N}
\sum_{K\in\calD}\mu(Q_K^{\mathrm{up}})a_K(X)^{p/2}
+\mu(\D)\varepsilon^{p/2}.
\end{align}
Combining \eqref{20261020eq20}, \eqref{20261020eq21} and the fact that the sequence $h$ constructed in \eqref{20261020eq02} satisfies the constraint in \eqref{20261020eq01}, we conclude that
$$
\begin{aligned}
&\inf_{\substack{h_K>0\\
\sum_{K\in\calD}\mu(Q_K^{\mathrm{up}})h_K^{p/(p-2)}\le1}}
\left[ \sup_{X\in\mathcal M_N}
\sum_{K\in\calD}\mu(Q_K^{\mathrm{up}})h_Ka_K(X) \right] \\
&\qquad\le \left(
\sup_{X\in\mathcal M_N}
\sum_{K\in\calD}\mu(Q_K^{\mathrm{up}})a_K(X)^{p/2}+\mu(\D)\varepsilon^{p/2}
\right)^{2/p}.
\end{aligned}
$$
This holds for every $\varepsilon>0$. Letting $\varepsilon \to 0^{+}$ proves \eqref{20261020eq30}.

\subsubsection{The case $p=2$}

Taking $h_K=1$ in \eqref{20261019eq05} and applying \eqref{20261019eq07} with $p=2$, we obtain
\begin{align*}
&\inf\left\{\sum_{I\in\calD_{\le N}}\lambda_I:
\lambda_I\ge0,\quad \Gamma_{\mu,N}\le\operatorname{diag}(\lambda_I)\right\}=\sup_{\substack{X\ge0\\X_{II}\le1}}
\sum_{K\in\calD}\mu(Q_K^{\mathrm{up}})a_K(X)
\lesssim C_{\calD,2}(\mu).
\end{align*}
Taking the supremum over $N$ proves
$$
\calB_{\calD,2}(\mu)\lesssim C_{\calD,2}(\mu).
$$

\subsubsection{The case $p>2$}

Let $h=(h_K)_{K\in\calD}$ be a nonnegative sequence satisfying
\begin{equation} \label{20261021eq40}
\sum_{K\in\calD}\mu(Q_K^{\mathrm{up}})h_K^{p/(p-2)}\le1.
\end{equation}
For every $X\ge0$ with $X_{II}\le1$, by H\"older, we have 
\begin{align*}
\sum_{K\in\calD}\mu(Q_K^{\mathrm{up}})h_Ka_K(X)
&\le\left(\sum_{K\in\calD}\mu(Q_K^{\mathrm{up}})h_K^{p/(p-2)}\right)^{(p-2)/p}
\left(\sum_{K\in\calD}\mu(Q_K^{\mathrm{up}})a_K(X)^{p/2}\right)^{2/p}\\
&\le\left(\sum_{K\in\calD}\mu(Q_K^{\mathrm{up}})a_K(X)^{p/2}\right)^{2/p}\\
&\lesssim C_{\calD,p}(\mu),
\end{align*}
where the last inequality follows from \eqref{20261019eq07}. Taking the supremum over $X$ and using \eqref{20261019eq05}, we obtain
$$
\inf\left\{\sum_{I\in\calD_{\le N}}\lambda_I:
\lambda_I\ge0,\quad \Gamma_{\mu,h,N}\le\operatorname{diag}(\lambda_I)\right\}
\lesssim C_{\calD,p}(\mu).
$$
Finally, taking the supremum over all $h$ satisfying \eqref{20261021eq40} and then over $N$ we deduce that 
$$
\calB_{\calD,p}(\mu)\lesssim C_{\calD,p}(\mu)
$$
for the case $p>2$.
\medskip 

The proof of \eqref{20261020mainest01} is complete. Hence, by Proposition~\ref{20261016prop01} and Proposition~\ref{20261018prop01}, it follows that
\begin{align*}
\|\operatorname{id}:\calB\to L^p(\mu)\|
&\simeq\|\operatorname{id}:\calB\to L^p(\mu_{\calD})\|\\
&\simeq\mu(\D)^{1/p}+C^{1/2}_{\calD,p}(\mu)\\
&\simeq\mu(\D)^{1/p}+\calB^{1/2}_{\calD,p}(\mu), 
\end{align*}
which completes the proof of Theorem~\ref{boundedness}.

\medskip 

\section{Compactness}\label{Sec05}

In this section, we prove Theorem~\ref{thm:compact}. Recall that $S_\rho=\{z\in\D:|z|>\rho\}$ for $0<\rho<1$. We first record the following compactness criterion.

\begin{lem}\label{lem:compactness}
Let $0<p<\infty$ and suppose that \eqref{20261015eq01} is bounded. Then this embedding is compact if and only if every bounded sequence $(f_j)$ in $\calB$ converging to zero uniformly on compact subsets of $\D$ satisfies
$$
\int_\D |f_j(z)|^p\,d\mu(z)\longrightarrow0.
$$
\end{lem}

\begin{proof}
The proof is standard and is therefore omitted.
\end{proof}

\begin{proof}[Proof of Theorem~\ref{thm:compact}]
Suppose first that \eqref{20261015eq01} is compact. Then it is bounded, and hence Theorem~\ref{boundedness} gives $\calB_{\calD,p}(\mu)<\infty$. To prove the vanishing condition in \eqref{20261021cond01}, we first show that
\begin{equation}\label{20261021eq01}
\lim_{\rho\to1^-}\|\operatorname{id}:\calB\to L^p(\one_{S_\rho}\mu)\|=0.
\end{equation}

If \eqref{20261021eq01} fails, there exist $\varepsilon_0>0$, a sequence $\rho_j\to1^-$, and functions $f_j\in\calB$ such that
\begin{equation}\ \label{202610321eq02}
\|f_j\|_{\calB}\le1,
\qquad \textrm{and} \qquad
\int_{S_{\rho_j}}|f_j(z)|^p\,d\mu(z)\ge\varepsilon_0.
\end{equation}
Since \eqref{20261015eq01} is compact, one can find a subsequence of $(f_j)$, still denoted by $(f_j)$, converging to some $f$ in $L^p(\mu)$. Consequently,
$$
\int_{S_{\rho_j}}|f_j(z)|^p\,d\mu(z)
\lesssim
\int_\D |f_j(z)-f(z)|^p\,d\mu(z)
+\int_{S_{\rho_j}}|f(z)|^p\,d\mu(z)
\longrightarrow0,
$$
This contradicts \eqref{202610321eq02} and hence \eqref{20261021eq01} holds. 

To this end, apply \eqref{20261016eq05} to $\one_{S_\rho}\mu$ and we obtain
$$
\calB_{\calD,p}(\one_{S_\rho}\mu)
\lesssim
\|\operatorname{id}:\calB\to L^p(\one_{S_\rho}\mu)\|^2
\longrightarrow0
\qquad(\rho\to1^-).
$$
Thus \eqref{20261021cond01} holds.

\vspace{0.1cm}

Conversely, suppose that \eqref{20261021cond01} holds. Then by Theorem~\ref{boundedness}, the embedding  \eqref{20261015eq01} is bounded. Since $\mu$ is finite, $\mu(S_\rho)\to0$ as $\rho\to1^-$. Applying \eqref{20261016eq05} to $\one_{S_\rho}\mu$, we obtain
\begin{equation}\label{20261021eq14}
\|\operatorname{id}:\calB\to L^p(\one_{S_\rho}\mu)\|^2
\lesssim
\mu(S_\rho)^{2/p}+\calB_{\calD,p}(\one_{S_\rho}\mu)
\longrightarrow0
\qquad \textrm{as} \qquad \rho\to1^-.
\end{equation}
By Lemma~\ref{lem:compactness}, to show the compactness of \eqref{20261015eq01}, it suffices to show 
\begin{equation} \label{20261021eq20}
\int_\D |f_j(z)|^p\,d\mu(z)\longrightarrow0
\end{equation}
where $(f_j)$ is any bounded sequence in $\calB$ and converges to zero uniformly on compact subsets of $\D$. We may also assume that $\|f_j\|_{\calB}\le1$ for every $j$. Then for each fixed $0<\rho<1$, we have
\begin{align*}
\int_\D |f_j(z)|^p\,d\mu(z)
&=\int_{\{|z|\le\rho\}}|f_j(z)|^p\,d\mu(z)
+\int_{S_\rho}|f_j(z)|^p\,d\mu(z)\\
&\le\mu(\D) \cdot \sup_{|z|\le\rho}|f_j(z)|^p
+\|\operatorname{id}:\calB\to L^p(\one_{S_\rho}\mu)\|^p.
\end{align*}
The first term tends to zero as $j\to\infty$. Hence
$$
\limsup_{j\to\infty}\int_\D |f_j(z)|^p\,d\mu(z)
\le\|\operatorname{id}:\calB\to L^p(\one_{S_\rho}\mu)\|^p.
$$
The desired claim \eqref {20261021eq20} therefore follows by letting $\rho\to1^-$ in the above estimate and using \eqref{20261021eq14}. This completes the proof of Theorem~\ref{thm:compact}. 
\end{proof}

\bigskip

\section{Classical logarithmic conditions and atomic examples}\label{Sec06}

The purpose of this section is to compare the $p$-Carleson Bloch capacity with the classical logarithmic conditions \eqref{20261016cond01} and \eqref{20261016cond02} of \cite{GPPR}; for simplicity, we restrict attention to $p=2$.

\begin{enumerate}
    \item We bound $\mu(\D)+\calB_{\calD,2}(\mu)$ between the first and square logarithmic moments, up to constants.

    \item We compute the capacity for a family of atomic measures, up to constants.

    \item We give two examples showing that the condition $\calB_{\calD, 2}(\mu)<+\infty$ is not equivalent to either of the classical logarithmic conditions:
    \begin{enumerate}
        \item A measure with finite capacity but an infinite square logarithmic moment.
        \item A measure with infinite capacity but a finite first logarithmic moment.
    \end{enumerate}
\end{enumerate}

\subsection{Comparison with logarithmic moments}

We first prove the comparisons stated in \eqref{20261025eq40}.

\begin{prop}\label{20261022prop10}
For every finite positive Borel measure $\mu$ on $\D$, we have
\begin{equation}\label{eq:moment-comparison}
\int_\D\log\frac{e}{1-|z|}\,d\mu(z)
\lesssim \mu(\D)+\calB_{\calD,2}(\mu)
\lesssim \int_\D\left(\log\frac{e}{1-|z|}\right)^2\,d\mu(z).
\end{equation}
\end{prop}

\begin{proof}
We first prove the upper bound. Recall that for every $f\in\calB$, one has 
$$
|f(z)|\le |f(0)|+\frac{\|f\|_{\calB,*}}{2}\log\frac{1+|z|}{1-|z|}
\lesssim \|f\|_{\calB}\log\frac{e}{1-|z|}.
$$
Taking the supremum over $\|f\|_{\calB}\le1$ and using \eqref{20261016eq05}, we obtain
$$
\mu(\D)+\calB_{\calD,2}(\mu)
\simeq \|\operatorname{id}:\calB\to L^2(\mu)\|^2
\lesssim \int_\D\left(\log\frac{e}{1-|z|}\right)^2\,d\mu(z).
$$

We next prove the lower bound. We may assume that $\calB_{\calD,2}(\mu)<\infty$. Let $\varepsilon=(\varepsilon_j)_{j\ge0}$ be a sequence of i.i.d. Rademacher random variables. For each integer $L\ge0$, define
$$
F_{L,\varepsilon}(z)=\sum_{j=0}^L\varepsilon_j z^{2^j}.
$$
Observe that these polynomials have uniformly bounded Bloch norms. Since $F_{L,\varepsilon}(0)=0$, this follows from the estimate
\begin{align*}
(1-|z|^2)|F_{L,\varepsilon}'(z)|
&\le (1-|z|^2)\sum_{j=0}^L2^j|z|^{2^j-1}\\
&\lesssim (1-|z|^2)\sum_{j\ge0}2^j e^{-2^j(1-|z|)}\\
&\lesssim (1-|z|^2)\int_{1/2}^{\infty}e^{-s(1-|z|)}\,ds
\lesssim\frac{1-|z|^2}{1-|z|}\lesssim1.
\end{align*}
Applying \eqref{20261015eq01} with $p=2$ to each $F_{L,\varepsilon}$ and using the uniform bound on its Bloch norm, we obtain
\begin{equation}  \label{20261021eq15}
\int_\D|F_{L,\varepsilon}(z)|^2\,d\mu(z)
\le \|\operatorname{id}:\calB\to L^2(\mu)\|^2\|F_{L,\varepsilon}\|_{\calB}^2
\lesssim \|\operatorname{id}:\calB\to L^2(\mu)\|^2.
\end{equation}
Hence, taking the average over the signs, we have 
\begin{align*}
\frac{1}{2^{L+1}}\sum_{\varepsilon\in\{-1,1\}^{L+1}}|F_{L,\varepsilon}(z)|^2
=\sum_{j,k=0}^L z^{2^j}\overline{z}^{\,2^k}
\left(\frac{1}{2^{L+1}}\sum_{\varepsilon\in\{-1,1\}^{L+1}}\varepsilon_j\varepsilon_k\right)=\sum_{j=0}^L|z|^{2^{j+1}}, 
\end{align*}
where in the last estimate above, we used the fact that 
$$
\frac{1}{2^{L+1}}\sum_{\varepsilon\in\{-1,1\}^{L+1}}\varepsilon_j\varepsilon_k
=\begin{cases}
1,&j=k,\\
0,&j\ne k.
\end{cases}
$$
Integrating this identity over $\D$ and using \eqref{20261021eq15}, we conclude that
\begin{align*}
\int_\D\sum_{j=0}^L|z|^{2^{j+1}}\,d\mu(z)
&=\frac{1}{2^{L+1}}\sum_{\varepsilon\in\{-1,1\}^{L+1}}
\int_\D|F_{L,\varepsilon}(z)|^2\,d\mu(z)\\
&\lesssim \|\operatorname{id}:\calB\to L^2(\mu)\|^2.
\end{align*}
Letting $L\to\infty$, we see that
$$
\int_\D\sum_{j\ge0}|z|^{2^{j+1}}\,d\mu(z)
\lesssim \|\operatorname{id}:\calB\to L^2(\mu)\|^2.
$$
Hence, 
\begin{align*}
\int_\D\log\frac{e}{1-|z|}\,d\mu(z)
&\lesssim \mu(\D)+\int_\D\sum_{j\ge0}|z|^{2^{j+1}}\,d\mu(z)\\
&\lesssim \mu(\D)+\|\operatorname{id}:\calB\to L^2(\mu)\|^2\\
&\simeq \mu(\D)+\calB_{\calD,2}(\mu),
\end{align*}
where we used \eqref{20261016eq05} and the elementary inequality
$$
\log\frac{e}{1-r}\lesssim 1+\sum_{j\ge0}r^{2^{j+1}},
\qquad 0\le r<1.
$$
This completes the proof.
\end{proof}

\subsection{Computing the capacity for atomic measures}

We give an example on an explicit computation for the $2$-Carleson Bloch capacity, up to constants, for a family of atomic measures. For $n\ge1$, define
\begin{equation}\label{eq:circle-measures}
\sigma_n:=2^{-n}\sum_{I\in\calD_n}\delta_{z_I},
\end{equation}
where $z_I$ is given by \eqref{20261016defn01X}. Observe that $\sigma_n$ is a probability measure. 

We have the following result. 

\begin{prop}\label{20261022prop01}
Let $(a_n)_{n\ge1}$ be a nonnegative sequence satisfying $\sum_{n\ge1}a_n<\infty$, and define
\begin{equation}\label{eq:weighted-circles}
\mu:=\sum_{n\ge1}a_n\sigma_n.
\end{equation}
Then
\begin{equation}\label{20261022eq50}
\calB_{\calD,2}(\mu)
\simeq\|\operatorname{id}:\calB\to L^2(\mu)\|^2
\simeq\sum_{n\ge1}n a_n.
\end{equation}
Moreover, for every $0<\rho<1$,
\begin{equation}\label{20261022eq51}
\calB_{\calD,2}(\one_{S_\rho}\mu)
\simeq\sum_{n\ge1, \; 1-\frac{3\cdot2^{-n}}4>\rho}n a_n.
\end{equation}
In particular, for $\mu$ defined by \eqref{eq:weighted-circles}, the embedding $\operatorname{id}:\calB\to L^2(\mu)$ is bounded if and only if it is compact.
\end{prop}

For the upper bound in \eqref{20261022eq50}, we shall use the following estimate.

\begin{lem}\label{lem:circle-average}
For every $f\in\calB$ and $0<r<1$, we have
\begin{equation}\label{20261022eq52}
\int_0^1|f(re^{2\pi it})|^2\,dt
\le |f(0)|^2+\|f\|_{\calB,*}^2\log\frac{1}{1-r^2}.
\end{equation}
Consequently, for every $n\ge1$,
\begin{equation}\label{20261022eq53}
\int_\D|f(z)|^2\,d\sigma_n(z)
=2^{-n}\sum_{I\in\calD_n}|f(z_I)|^2
\lesssim n\|f\|_{\calB}^2.
\end{equation}
\end{lem}

\begin{proof}
We first prove \eqref{20261022eq52}. Writing $f(z)=\sum_{k\ge0}b_kz^k$, we first apply Parseval's identity to $f'$ to obtain
$$
\int_0^1|f'(se^{2\pi it})|^2\,dt
=\sum_{k\ge1}k^2|b_k|^2s^{2k-2},
\qquad 0<s<r.
$$
Applying Parseval's identity to $f$, we obtain
\begin{align*}
\int_0^1|f(re^{2\pi it})|^2\,dt-|f(0)|^2
&=\sum_{k\ge1}|b_k|^2r^{2k}\\
&=4\int_0^r s\log\frac{r}{s}
\left(\sum_{k\ge1}k^2|b_k|^2s^{2k-2}\right)\,ds\\
&=4\int_0^r s\log\frac{r}{s}
\left(\int_0^1|f'(se^{2\pi it})|^2\,dt\right)\,ds\\
&\le4\|f\|_{\calB,*}^2\int_0^r\frac{s\log(r/s)}{(1-s^2)^2}\,ds\\
&=\|f\|_{\calB,*}^2\log\frac{1}{1-r^2}.
\end{align*}
Here the second equality follows from
$$
\int_0^r s^{2k-1}\log\frac{r}{s}\,ds=\frac{r^{2k}}{4k^2}.
$$
For the last equality, writing $\log(r/s)=\int_s^r du/u$, we have
\begin{align*}
4\int_0^r\frac{s\log(r/s)}{(1-s^2)^2}\,ds
=4\int_0^r\frac1u\left(\int_0^u\frac{s}{(1-s^2)^2}\,ds\right)\,du=2\int_0^r\frac{u}{1-u^2}\,du
=\log\frac{1}{1-r^2}.
\end{align*}
This proves \eqref{20261022eq52}.

\vspace{0.1cm}

To prove \eqref{20261022eq53}, observe that for $e^{2\pi it}\in I\in\calD_n$, the point $(1-\frac{3\cdot2^{-n}}4)e^{2\pi it}$ belongs to $Q_I^{\mathrm{up}}$. Hence, by  Lemma~\ref{20261016lem01}, we have 
$$
|f(z_I)|^2\lesssim
\left|f\left(\left(1-\frac{3 \cdot 2^{-n}}{4} \right)e^{2\pi it}\right)\right|^2
+\|f\|_{\calB,*}^2.
$$
Therefore, 
\begin{align*}
2^{-n}\sum_{I\in\calD_n}|f(z_I)|^2
&=\sum_{I\in\calD_n}\int_I|f(z_I)|^2\,dt\\
&\lesssim\sum_{I\in\calD_n}\int_I
\left[\left|f\left(\left(1-\frac{3\cdot2^{-n}}{4}\right)e^{2\pi it}\right)\right|^2+\|f\|_{\calB,*}^2\right]\,dt\\
&=\int_0^1\left|f\left(\left(1-\frac{3\cdot2^{-n}}{4}\right)e^{2\pi it}\right)\right|^2\,dt+\|f\|_{\calB,*}^2\\
&\lesssim |f(0)|^2+n\|f\|_{\calB,*}^2
\lesssim n\|f\|_{\calB}^2,
\end{align*}
where we used \eqref{20261022eq52} and the fact that
$$
\log\frac{1}{1-(1-\frac{3\cdot2^{-n}}{4})^2}\lesssim n,
\qquad n\ge1.
$$
This completes the proof.
\end{proof}

\begin{proof}[Proof of Proposition~\ref{20261022prop01}]
We first prove the upper bound in \eqref{20261022eq50}. By Lemma~\ref{lem:circle-average}, for every $f\in\calB$,
$$
\int_\D|f(z)|^2\,d\mu(z)
=\sum_{n\ge1}a_n\int_\D|f(z)|^2\,d\sigma_n(z)
\lesssim\left(\sum_{n\ge1}n a_n\right)\|f\|_{\calB}^2.
$$
Taking the supremum over $\|f\|_{\calB}\le1$ and using \eqref{20261016eq05}, we obtain
$$
\calB_{\calD,2}(\mu)
\lesssim\|\operatorname{id}:\calB\to L^2(\mu)\|^2
\lesssim\sum_{n\ge1}n a_n.
$$
For the reverse bound, observe first that 
\begin{equation} \label{20261023eq01X}
\log\frac{e}{1-|z|} \simeq n, \qquad z \in \supp \; \sigma_n. 
\end{equation}
Hence, by Proposition~\ref{20261022prop10}, we have 
\begin{align*}
\sum_{n\ge1}n a_n \simeq\int_\D\log\frac{e}{1-|z|}\,d\mu(z)\lesssim\mu(\D)+\calB_{\calD,2}(\mu).
\end{align*}
Therefore, it remains to show that $\mu(\D)\lesssim\calB_{\calD,2}(\mu)$.
We shall obtain this estimate by taking $f(z)=z$ in \eqref{20261018eq20}. For this purpose, we first compute $\mu(Q_I^{\mathrm{up}})$ for every $I\in\calD$.

For $I\in\calD_k$, $k\ge1$, the tent $Q_I$ contains exactly $2^{n-k}$ atoms of $\sigma_n$ when $n\ge k$, and none when $n<k$. Therefore,
$$
\mu(Q_I)=2^{-k}\sum_{n\ge k}a_n,
$$
and hence if $I_+$ and $I_-$ are the two dyadic children of $I$, then
$$
\mu(Q_I^{\mathrm{up}})
=\mu(Q_I)-\mu(Q_{I_+})-\mu(Q_{I_-})
=2^{-k}a_k.
$$
Also, $\mu(Q_{\T}^{\mathrm{up}})=0$.

\vspace{0.1cm}

Since $f(z)=z$ belongs to $\mathring{\calB}$ and has Bloch seminorm $1$, we obtain
\begin{align*}
\mu(\D)
&=\sum_{k\ge1}a_k \lesssim\sum_{k\ge1}a_k\left(1-\frac{3\cdot2^{-k}}4\right)^2=\sum_{k\ge1}2^{-k}a_k\left(\sum_{I\in\calD_k}|z_I|^2\right)\\
&=\sum_{I\in\calD}\mu(Q_I^{\mathrm{up}})|z_I|^2\lesssim C_{\calD,2}(\mu)
\le\calB_{\calD,2}(\mu),
\end{align*}
where we used \eqref{20261018eq20},  \eqref{20261020mainest01} and the fact that $|z_I|=1-\frac{3\cdot2^{-k}}4, \; I \in \calD_k$ in the above estimates. The proof of \eqref{20261022eq50} is complete. 

\vspace{0.1cm}

We next prove \eqref{20261022eq51}. Since each $\sigma_n$ is supported on the circle $|z|=1-\frac{3\cdot2^{-n}}4$, we have
$$
\one_{S_\rho}\mu
=\sum_{n\ge1, \; 1-\frac{3\cdot2^{-n}}4>\rho}a_n\sigma_n.
$$
Hence, by \eqref{20261022eq50}, we have 
$$
\calB_{\calD,2}(\one_{S_\rho}\mu)
\simeq\sum_{n\ge1, \; 1-\frac{3\cdot2^{-n}}4>\rho}n a_n,
$$
which proves \eqref{20261022eq51}.

\vspace{0.1cm}

Finally, suppose that $\operatorname{id}:\calB\to L^2(\mu)$ is bounded. By \eqref{20261022eq50}, we have $\sum_{n\ge1}n a_n<\infty$. Hence,
$$
\lim_{\rho \to 1^{-}} \calB_{\calD,2}(\one_{S_\rho}\mu) =\lim_{\rho \to 1^{-}}  \sum_{n\ge1, \; 1-\frac{3\cdot2^{-n}}4>\rho}n a_n=0, 
$$
which, by Theorem~\ref{thm:compact} implies that the desired embedding is compact. The converse is immediate.
\end{proof}

\subsection{Two examples}

Finally, we give two examples showing that, for $p=2$, neither of the logarithmic conditions \eqref{20261016cond01} and \eqref{20261016cond02} is equivalent to $\calB_{\calD,2}(\mu)<\infty$.

\subsubsection{A measure that has finite $2$--Carleson Bloch capacity and an infinite square logarithmic moment} 

\begin{cor}\label{20261023Cor01}
Define
$$
\mu_+:=\sum_{n\ge1}\frac{1}{n^3}\sigma_n
=\sum_{n\ge1}\frac{1}{n^3} \left(\frac{1}{ 2^n}\sum_{I\in\calD_n}\delta_{z_I} \right),
$$
where we recall that $\sigma_n$ is defined in \eqref{eq:circle-measures}. Then $\mu_+$ is a finite positive Borel measure satisfying
\begin{equation}\label{20261023eq01}
\begin{gathered}
\calB_{\calD,2}(\mu_+)<\infty,
\qquad
\int_\D\log\frac{e}{1-|z|}\,d\mu_+(z)<\infty, \qquad \textrm{and}\\
 \qquad \int_\D\left(\log\frac{e}{1-|z|}\right)^2\,d\mu_+(z)=\infty.
\end{gathered}
\end{equation}
Moreover, the embedding $\operatorname{id}:\calB\to L^2(\mu_+)$ is compact, and
\begin{equation}\label{20261023eq02}
\calB_{\calD,2}(\one_{S_\rho}\mu_+)
\simeq\left(\log\frac{e}{1-\rho}\right)^{-1},
\qquad 0<\rho<1.
\end{equation}
\end{cor}

\begin{proof}
The proof of the above result is a direct application of Proposition \ref{20261022prop01}.  Since each $\sigma_n$ is a probability measure, we have $\mu_+(\D)=\sum_{n\ge1}1/n^3<\infty$. Now taking $a_n=n^{-3}$ in \eqref{20261022eq50}, we obtain
$$
\calB_{\calD,2}(\mu_+)\simeq\sum_{n\ge1}\frac{1}{n^2}<\infty.
$$
To compute the logarithmic integrals, recall that by \eqref{20261023eq01X}, 
$$
\log\frac{e}{1-|z|}\simeq n,
\qquad z\in\supp \; \sigma_n.
$$
Consequently,
$$
\int_\D\log\frac{e}{1-|z|}\,d\mu_+(z) \simeq\sum_{n\ge1}\frac{1}{n^2}<\infty,
$$
and
$$
\int_\D\left(\log\frac{e}{1-|z|}\right)^2\,d\mu_+(z) \simeq\sum_{n\ge1}\frac{1}{n}=\infty.
$$
This proves \eqref{20261023eq01}.

\vspace{0.1cm}

We next prove \eqref{20261023eq02}. By \eqref{20261022eq51}, we have
\begin{align*}
\calB_{\calD,2}(\one_{S_\rho}\mu_+)
&\simeq\sum_{n\ge1,\;1-\frac{3\cdot2^{-n}}4>\rho}\frac{1}{n^2} \simeq\int_{\max\{1,\log_2\frac{3}{4(1-\rho)}\}}^\infty\frac{dt}{t^2}\simeq\left(\log\frac{e}{1-\rho}\right)^{-1},
\end{align*}
which further gives 
$$
\lim_{\rho\to1^-}\calB_{\calD,2}(\one_{S_\rho}\mu_+)=0.
$$
This implies the compactness of $\operatorname{id}:\calB\to L^2(\mu_+)$ by Theorem~\ref{thm:compact}.
\end{proof}

\subsubsection{A measure with infinite $2$--Carleson Bloch capacity and a finite first logarithmic moment}

\begin{prop}
Fix $\xi\in\T$ and define
$$
\mu_-:=\sum_{n\ge1}\frac{1}{n^3}\delta_{(1-\frac{3\cdot2^{-n}}4)\xi}.
$$
Then $\mu_-$ is a finite positive Borel measure satisfying
$$
\int_\D\log\frac{e}{1-|z|}\,d\mu_-(z)<\infty,
\qquad \textrm{and} \qquad 
\calB_{\calD,2}(\mu_-)=\infty.
$$
In particular, the embedding $\operatorname{id}:\calB\to L^2(\mu_-)$ is unbounded.
\end{prop}

\begin{proof}
By the definition of $\mu_-$, we have
$$
\mu_-(\D)=\sum_{n\ge1}\frac{1}{n^3}<\infty,
$$
and
$$
\int_\D\log\frac{e}{1-|z|}\,d\mu_-(z)
\simeq\sum_{n\ge1}\frac{1}{n^2}<\infty.
$$

It remains to prove that $\calB_{\calD,2}(\mu_-)=\infty$. By Theorem~\ref{boundedness}, it suffices to show that $\operatorname{id}:\calB\to L^2(\mu_-)$ is unbounded. Consider
$$
f_\xi(z):=\log\frac{1}{1-\overline{\xi}z},
\qquad z\in\D.
$$
It is clear that $f_\xi \in \mathring{\calB}$. Observe that 
\begin{align*}
\int_\D|f_\xi(z)|^2\,d\mu_-(z)
&=\sum_{n\ge1}\frac{1}{n^3}
\left|f_\xi\left(\left(1-\frac{3\cdot2^{-n}}4\right)\xi\right)\right|^2\\
&=\sum_{n\ge1}\frac{1}{n^3}\left(\log\frac43+n\log2\right)^2\\
&\simeq\sum_{n\ge1}\frac{1}{n}=\infty.
\end{align*}
Thus the embedding is unbounded, which proves the remaining assertion.
\end{proof}

\medskip 

\section{An extension to $\calQ_p$ spaces}\label{Sec07}

The purpose of this section is to extend the preceding characterization to the embedding
$$
\operatorname{id}:\calQ_p\longrightarrow L^q(\mu), \qquad 0<p\le1,\quad 0<q<\infty.
$$
As in \cite[Section~6]{HZBloch}, the main modification is to incorporate the local estimates defining the $\calQ_p$ seminorm into the coefficient constraints. We retain the polynomials $\Phi_I$ and the matrices $\Gamma_{\mu,h,N}$ defined in \eqref{20261017eq05} and \eqref{20261023eq34}, respectively.

For $0<p\le1$, we use the equivalent seminorm
\begin{equation}\label{20261024eq01}
\|f\|_{\calQ_p,*}^2:=\sup_{J\subseteq\T}\frac{1}{|J|^p}
\int_{Q_J}|f'(z)|^2(1-|z|^2)^p\,dA(z),
\qquad
\|f\|_{\calQ_p}:=|f(0)|+\|f\|_{\calQ_p,*},
\end{equation}
where the supremum is taken over all nonempty arcs of $\T$. Recall that $\calQ_p\subseteq\calB$ for $0<p\le1$, $\calQ_1=\mathrm{BMOA}$, and $\calQ_p=\calB$ for $p>1$.

\subsection{The $q$--Carleson $\calQ_p$ capacity and the embedding theorem}

\begin{defn}
Let $0<p\le1$ and let $\mu$ be a finite positive Borel measure on $\D$. For $0<q<\infty$, we define its \emph{$q$-Carleson $\calQ_p$ capacity} $\mathfrak{Q}_{\calD,p,q}(\mu)$ as follows.
\begin{enumerate}
    \item[$\bullet$] For $0<q<2$, set
    \begin{equation}\label{20261024eq02}
    \begin{aligned}
    \mathfrak{Q}_{\calD,p,q}(\mu)
    :=&\sup_{N\ge0}\ \inf_{\substack{h_K>0\\
    \sum_{K\in\calD}\mu(Q_K^{\mathrm{up}})h_K^{q/(q-2)}\le1}}\\
    &\quad\inf\left\{
    \sum_{J\in\calD_{\le N}}\lambda_J|J|^p:
    \begin{array}{l}
    \lambda_J\ge0,\quad J\in\calD_{\le N},\\[2pt]
    \displaystyle\Gamma_{\mu,h,N}\le
    \operatorname{diag}_{I\in\calD_{\le N}}
    \left(|I|^p\sum_{\substack{J\in\calD_{\le N}\\I\subseteq J}}\lambda_J\right)
    \end{array}
    \right\}.
    \end{aligned}
    \end{equation}

    \item[$\bullet$] For $q=2$, set
    \begin{equation}\label{20261024eq02a}
    \mathfrak{Q}_{\calD,p,2}(\mu):=\sup_{N\ge0}
    \inf\left\{
    \sum_{J\in\calD_{\le N}}\lambda_J|J|^p:
    \begin{array}{l}
    \lambda_J\ge0,\quad J\in\calD_{\le N},\\[2pt]
    \displaystyle\Gamma_{\mu,N}\le
    \operatorname{diag}_{I\in\calD_{\le N}}
    \left(|I|^p\sum_{\substack{J\in\calD_{\le N}\\I\subseteq J}}\lambda_J\right)
    \end{array}
    \right\}.
    \end{equation}

    \item[$\bullet$] For $2<q<\infty$, set
    \begin{equation}\label{20261024eq02b}
    \begin{aligned}
    \mathfrak{Q}_{\calD,p,q}(\mu)
    :=&\sup_{N\ge0}\ \sup_{\substack{h_K\ge0\\
    \sum_{K\in\calD}\mu(Q_K^{\mathrm{up}})h_K^{q/(q-2)}\le1}}\\
    &\quad\inf\left\{
    \sum_{J\in\calD_{\le N}}\lambda_J|J|^p:
    \begin{array}{l}
    \lambda_J\ge0,\quad J\in\calD_{\le N},\\[2pt]
    \displaystyle\Gamma_{\mu,h,N}\le
    \operatorname{diag}_{I\in\calD_{\le N}}
    \left(|I|^p\sum_{\substack{J\in\calD_{\le N}\\I\subseteq J}}\lambda_J\right)
    \end{array}
    \right\}.
    \end{aligned}
    \end{equation}
\end{enumerate}
If a diagonal entry of $\Gamma_{\mu,h,N}$ is infinite for some $h$ and $N$, the corresponding innermost infimum is understood to be $+\infty$. 
\end{defn}

\begin{thm}\label{20261024thm01}
Let $0<p\le1$, $0<q<\infty$, and let $\mu$ be a finite positive Borel measure on $\D$. Then the embedding
$$
\operatorname{id}:\calQ_p\longrightarrow L^q(\mu)
$$
is bounded if and only if $\mathfrak{Q}_{\calD,p,q}(\mu)<\infty$. Moreover,
\begin{equation}\label{20261024eq03}
\|\operatorname{id}:\calQ_p\to L^q(\mu)\|
\simeq \mu(\D)^{1/q}+\mathfrak{Q}_{\calD,p,q}(\mu)^{1/2}.
\end{equation}
The embedding is compact if and only if
$$
\mathfrak{Q}_{\calD,p,q}(\mu)<\infty,
\qquad
\lim_{\rho\to1^-}\mathfrak{Q}_{\calD,p,q}(\one_{S_\rho}\mu)=0,
$$
where we recall that $S_\rho=\{z\in\D:|z|>\rho\}$.
\end{thm}

\subsection{The representation of $\calQ_p$ functions}

To prove Theorem~\ref{20261024thm01}, we first establish the following counterpart of Theorem~\ref{20261017thm01} in the $\calQ_p$--setting.

\begin{lem}\label{20261024lem01}
Let $0<p\le1$. Then the following assertions hold.
\begin{enumerate}
    \item Every sequence $(c_I)_{I\in\calD}$ satisfying
    $$
    \sup_{J\in\calD}\frac{1}{|J|^p}
    \sum_{\substack{I\in\calD\\I\subseteq J}}|c_I|^2|I|^p<\infty
    $$
    defines a function
    $$
    f=\sum_{I\in\calD}c_I\Phi_I\in\calQ_p,\qquad f(0)=0,
    $$
    with
    $$
    \|f\|_{\calQ_p,*}^2
    \lesssim
    \sup_{J\in\calD}\frac{1}{|J|^p}
    \sum_{\substack{I\in\calD\\I\subseteq J}}|c_I|^2|I|^p.
    $$
    The series converges absolutely and uniformly on compact subsets of $\D$.

    \item Every $f\in\calQ_p$ with $f(0)=0$ admits such a representation with
    $$
    \sup_{J\in\calD}\frac{1}{|J|^p}
    \sum_{\substack{I\in\calD\\I\subseteq J}}|c_I|^2|I|^p
    \lesssim \|f\|_{\calQ_p,*}^2.
    $$
\end{enumerate}
Consequently, for every $f\in\calQ_p$ with $f(0)=0$, we have
$$
\|f\|_{\calQ_p,*}^2
\simeq
\inf_{f=\sum_{I\in\calD}c_I\Phi_I}
\sup_{J\in\calD}\frac{1}{|J|^p}
\sum_{\substack{I\in\calD\\I\subseteq J}}|c_I|^2|I|^p.
$$
\end{lem}

\begin{proof}
We first prove (1) for finitely supported sequences. Without loss of generality, we may assume that
$$
\sup_{J\in\calD}\frac{1}{|J|^p}
\sum_{\substack{I\in\calD\\I\subseteq J}}|c_I|^2|I|^p\le1.
$$
Hence, $|c_I|\le1$ for every $I\in\calD$. Our goal is to prove
$$
\int_{Q_J}\left|\sum_{I\in\calD}c_I\Phi_I'(z)\right|^2
(1-|z|^2)^p\,dA(z)\lesssim |J|^p.
$$
We first establish two estimates for the polynomials $\Phi_I$.

\vspace{0.1cm}

\noindent\textit{Claim 1.} There exists $c>0$ such that for any $I \in \calD$ with the center $\xi_I$ and $1/2 \le r<1$, 
\begin{equation}\label{20261024eq05}
|\Phi_I(re^{2\pi it})|+|I||\Phi_I'(re^{2\pi it})|
\lesssim
\frac{e^{-c(1-r)/|I|}}{(1+d_\T(t,t_I)/|I|)^4},
\end{equation}
where $t_I \in \mathbb R / \mathbb Z$ such that $\zeta_I=e^{2\pi it_I}$, and  $d_\T(t,s):=\min_{j\in\mathbb Z}|t-s-j|$.

\medskip 

\noindent \textit{Proof of Claim 1.} 
We first prove an analogue of the pointwise estimate \eqref{20261017eq01}. Let $n \ge 1$ and $\psi\in C^4(\mathbb R)$ be supported in $[1/16,3/4]$. By \eqref{20261024defn01}, we have
$$
K_{2^n,\psi}(\zeta)=\sum_{k\in\mathbb Z}\psi(k2^{-n})\zeta^k,
\qquad \zeta\in\T.
$$
Since there are $\lesssim2^n$ nonzero terms in the summation, we have 
$|K_{2^n,\psi}(\zeta)|\lesssim2^n\|\psi\|_\infty$. On the other hand, for $\xi\in\T$, we have
\begin{align}
\left|(1-\xi)^4K_{2^n,\psi}(\xi)\right|
&=\left|\sum_{j=0}^4(-1)^j\binom{4}{j}
\cdot\left[\sum_{k\in\mathbb Z}\psi(k2^{-n})\xi^{k+j}\right]\right|\notag\\
&=\left|\sum_{j=0}^4(-1)^j\binom{4}{j}
\cdot\left[\sum_{k\in\mathbb Z}\psi((k-j)2^{-n})\xi^k\right]\right|\notag\\
&=\left|\sum_{k\in\mathbb Z}
\left[\sum_{j=0}^4(-1)^j\binom{4}{j}\psi((k-j)2^{-n})\right]\xi^k\right|\notag\\
&\le\sum_{\substack{k\in\mathbb Z\\2^{n-4}\le k\le2^n+4}}
\left|\sum_{j=0}^4(-1)^j\binom{4}{j}\psi((k-j)2^{-n})\right|\notag\\
&=\sum_{\substack{k\in\mathbb Z\\2^{n-4}\le k\le2^n+4}}
\left|\int_{[0,2^{-n}]^4}
\psi^{(4)}(k2^{-n}-u_1-u_2-u_3-u_4)\,du_1\,du_2\,du_3\,du_4\right|\notag\\
&\le\sum_{\substack{k\in\mathbb Z\\2^{n-4}\le k\le2^n+4}}
2^{-4n}\|\psi^{(4)}\|_\infty
\lesssim2^{-3n}\|\psi^{(4)}\|_\infty \nonumber .
\end{align}
Combining the above estimates, we obtain
\begin{equation}\label{20261024eq05a}
|K_{2^n,\psi}(\zeta)|
\lesssim
\bigl(\|\psi\|_\infty+\|\psi^{(4)}\|_\infty\bigr)
\frac{2^n}{(1+2^n|\zeta-1|)^4}.
\end{equation}

As in the proof of Lemma~\ref{20261024lem01X}, we apply \eqref{20261024eq05a} to the functions
$$
q_{n,r}(s)=\eta(s)r^{2^ns},
\qquad
\widetilde q_{n,r}(s)=s\eta(s)r^{2^ns}.
$$
Direct computation gives
$$
q_{n,r}^{(4)}(s)=r^{2^ns}\sum_{j=0}^4\binom{4}{j}
\eta^{(j)}(s)(2^n\log r)^{4-j},
$$
and
$$
\widetilde q_{n,r}^{(4)}(s)
=r^{2^ns}\sum_{j=0}^4\binom{4}{j}
\frac{d^j}{ds^j}(s\eta(s))(2^n\log r)^{4-j}.
$$
Moreover, since both of these functions are supported in $[1/16,3/4]$, it follows that
\begin{align*}
&\|q_{n,r}\|_\infty+\|q_{n,r}^{(4)}\|_\infty
+\|\widetilde q_{n,r}\|_\infty+\|\widetilde q_{n,r}^{(4)}\|_\infty\\
&\qquad\lesssim
(1+2^n|\log r|)^4r^{2^n/16}
\lesssim e^{-2^n|\log r|/32}
\le e^{-2^n(1-r)/32}.
\end{align*}

For $\zeta=e^{2\pi it}$ and $I \in \calD_n, \; n \ge 1$, the definition of $\Phi_I$ gives
$$
\Phi_I(r\zeta)=2^{-n}K_{2^n,q_{n,r}}(\zeta\overline{\zeta_I}),
\qquad
|I|\Phi_I'(r\zeta)
=2^{-n}r^{-1}\overline{\zeta}\,
K_{2^n,\widetilde q_{n,r}}(\zeta\overline{\zeta_I}).
$$
Since $1/2 \le r<1$ and
$$
|\zeta\overline{\zeta_I}-1|
=2\sin\bigl(\pi d_\T(t,t_I)\bigr)
\simeq d_\T(t,t_I),
$$
we have \eqref{20261024eq05} follows from \eqref{20261024eq05a}. The case $I=\T$ is immediate because $\Phi_\T=0$.

\vspace{0.1cm}

\noindent\textit{Claim 2.} For every finitely supported sequence $(c_I)_{I\in\calD}$,
\begin{equation}\label{20261024eq06}
\int_\D\left|\sum_{I\in\calD}c_I\Phi_I'(z)\right|^2(1-|z|^2)^p\,dA(z)
\lesssim\sum_{I\in\calD}|c_I|^2|I|^p.
\end{equation}

\medskip 

\noindent \textit{Proof of Claim 2.} 
We first estimate the sum at each generation. Fix $n\ge1$. By \eqref{20261017eq05} and Parseval, we have
\begin{align*}
\int_0^1\left|\sum_{I\in\calD_n}c_I\Phi_I(e^{2\pi it})\right|^2\,dt
&=2^{-2n}\sum_{k\ge1}|\eta(k2^{-n})|^2
\left|\sum_{I\in\calD_n}c_I\overline{\zeta_I}^{\,k}\right|^2\\
&\lesssim2^{-2n}\sum_{k=0}^{2^n-1}
\left|\sum_{I\in\calD_n}c_I\overline{\zeta_I}^{\,k}\right|^2\\
&=2^{-n}\sum_{I\in\calD_n}|c_I|^2.
\end{align*}
Here we used the support condition on $\eta$ in the inequality. Moreover, for the last equality, we used the fact that
$$
\sum_{k=0}^{2^n-1}(\overline{\zeta_I}\zeta_J)^k
=\begin{cases}
2^n,&I=J,\\
0,&I\ne J, 
\end{cases}
$$
since the points $\zeta_I$, $I\in\calD_n$, form a rotation of the $2^n$th roots of unity. 

\vspace{0.1cm}

Next, by integration in polar coordinates and the substitution $u=r^2$, we obtain
$$
\int_\D |(z^k)'|^2(1-|z|^2)^p\,dA(z)
=2k^2\int_0^1r^{2k-1}(1-r^2)^p\,dr
=k^2\int_0^1u^{k-1}(1-u)^p\,du.
$$
For $k\ge2$, we have
\begin{align*}
k^2\int_0^1u^{k-1}(1-u)^p\,du
&\le k^2\int_0^\infty e^{-(k-1)v}v^p\,dv\\
&=\frac{k^2}{(k-1)^{p+1}}\int_0^\infty e^{-v}v^p\,dv
\lesssim k^{1-p}.
\end{align*}
For $k=1$, we obtain the same bound directly.

\vspace{0.1cm}

Moreover, by the support condition on $\eta$, we have $k\simeq2^n$, and hence for each fixed $k$, there are at most $\lesssim 1$ many $n$ for which $\eta(k2^{-n})\ne0$. Using Parseval again, we obtain
\begin{align*}
\int_\D\left|\sum_{I\in\calD}c_I\Phi_I'(z)\right|^2(1-|z|^2)^p\,dA(z) &\lesssim
\sum_{k\ge1}k^{1-p}
\left|\sum_{n\ge1}2^{-n}\eta(k2^{-n})
\left(\sum_{I\in\calD_n}c_I\overline{\zeta_I}^{\,k}\right)\right|^2\\
&\lesssim
\sum_{n\ge1}2^{-2n}\sum_{k\ge1}k^{1-p}|\eta(k2^{-n})|^2
\left|\sum_{I\in\calD_n}c_I\overline{\zeta_I}^{\,k}\right|^2\\
&\lesssim
\sum_{n\ge1}2^{n(1-p)}
\int_0^1\left|\sum_{I\in\calD_n}c_I\Phi_I(e^{2\pi it})\right|^2\,dt\\
&\lesssim
\sum_{n\ge1}2^{-np}\sum_{I\in\calD_n}|c_I|^2
\le\sum_{I\in\calD}|c_I|^2|I|^p.
\end{align*}
This proves \eqref{20261024eq06}.

\vspace{0.1cm}

\noindent\textit{Claim 3.} Let $(c_I)_{I\in\calD}$ be a finitely supported sequence satisfying
$$
\sup_{J\in\calD}\frac{1}{|J|^p}
\sum_{\substack{I\in\calD\\I\subseteq J}}|c_I|^2|I|^p\le1.
$$
Then
$$
\int_{Q_J}\left|\sum_{I\in\calD}c_I\Phi_I'(z)\right|^2(1-|z|^2)^p\,dA(z)
\lesssim |J|^p
$$
for every nonempty arc $J\subseteq\T$.

\medskip 

\noindent \textit{Proof of Claim 3.} 
First suppose that $J\in\calD$ and $|J|<1/4$. We split the sum into
$$
\sum_{I\in\calD}c_I\Phi_I'
=\sum_{|I|>|J|}c_I\Phi_I'
+\sum_{\substack{|I|\le|J|\\\zeta_I\in3J}}c_I\Phi_I'
+\sum_{\substack{|I|\le|J|\\\zeta_I\notin3J}}c_I\Phi_I'.
$$
Here and henceforth, for $b\ge1$, we denote by $bJ$ the concentric arc of length $\min\{b|J|,1\}$.

For the first sum, using $|c_I|\le1$ and \textit{Claim 1}, we obtain, for every $z=re^{2\pi it}\in Q_J$,
\begin{align*}
\left|\sum_{|I|>|J|}c_I\Phi_I'(z)\right|
&\le \sum_{\substack{n\ge1\\2^{-n}>|J|}}
\left(\sum_{I\in\calD_n}|\Phi_I'(z)|\right)\\
&\lesssim \sum_{\substack{n\ge1\\2^{-n}>|J|}}2^n
\left(\sum_{I\in\calD_n}\frac{1}{(1+2^nd_\T(t,t_I))^4}\right)\\
&\lesssim \sum_{\substack{n\ge1\\2^{-n}>|J|}}2^n
\left(\sum_{j\ge0}\frac{1}{(1+j)^4}\right)\\
&\lesssim \sum_{\substack{n\ge1\\2^{-n}>|J|}}2^n
\lesssim |J|^{-1},
\end{align*}
where we used $\Phi_\T=0$ in the first inequality and the equal spacing of the points $t_I$, $I\in\calD_n$, in the third inequality. Consequently, by integration in polar coordinates, we have
\begin{align*}
\int_{Q_J}\left|\sum_{|I|>|J|}c_I\Phi_I'(z)\right|^2(1-|z|^2)^p\,dA(z)
&\lesssim |J|^{-2}\int_{Q_J}(1-|z|^2)^p\,dA(z)\\
&\lesssim |J|^{-1}\int_{1-|J|}^1(1-r)^p\,dr\\
&\lesssim |J|^p.
\end{align*}

For the second sum, since $J$ is dyadic, $3J$ is the union of three adjacent dyadic arcs of length $|J|$. Every $I$ with $|I|\le|J|$ and $\zeta_I\in3J$ is contained in one of these arcs. Applying the coefficient constraint to these three arcs, we obtain
$$
\sum_{\substack{|I|\le|J|\\\zeta_I\in3J}}|c_I|^2|I|^p
\lesssim |J|^p.
$$
Thus, by \textit{Claim 2}, we have
\begin{align*}
&\int_{Q_J}\left|\sum_{\substack{|I|\le|J|\\\zeta_I\in3J}}
c_I\Phi_I'(z)\right|^2(1-|z|^2)^p\,dA(z)\\
&\qquad\le
\int_\D\left|\sum_{\substack{|I|\le|J|\\\zeta_I\in3J}}
c_I\Phi_I'(z)\right|^2(1-|z|^2)^p\,dA(z)\\
&\qquad\lesssim
\sum_{\substack{|I|\le|J|\\\zeta_I\in3J}}|c_I|^2|I|^p
\lesssim |J|^p.
\end{align*}

For the third sum, fix $z=re^{2\pi it}\in Q_J$. By \textit{Claim 1}, we have
\begin{align*}
\left|\sum_{\substack{|I|\le|J|\\\zeta_I\notin3J}}c_I\Phi_I'(z)\right|
&\lesssim \sum_{\substack{n\ge1\\2^{-n}\le|J|}}2^n
\left(\sum_{\substack{I\in\calD_n\\\zeta_I\notin3J}}
\frac{1}{(1+2^nd_\T(t,t_I))^4}\right)\\
&\lesssim \sum_{\substack{n\ge1\\2^{-n}\le|J|}}2^n
\left(\sum_{j\ge2^n|J|}\frac{1}{(1+j)^4}\right)\\
&\lesssim |J|^{-3}\sum_{\substack{n\ge1\\2^{-n}\le|J|}}2^{-2n}
\lesssim |J|^{-1}.
\end{align*}
Here in the second estimate above, we used the fact that $d_\T(t,t_I)\ge|J|$ for $\zeta_I\notin3J$. We can therefore estimate the contribution of this sum by
$$
\int_{Q_J}\left|\sum_{\substack{|I|\le|J|\\\zeta_I\notin3J}}c_I\Phi_I'(z)\right|^2(1-|z|^2)^p\,dA(z)
\lesssim |J|^{-2}\int_{Q_J}(1-|z|^2)^p\,dA(z)
\lesssim |J|^p.
$$
Combining this with the estimates for the first two sums, we conclude that the claim holds for every dyadic arc $J$ with $|J|<1/4$.

\vspace{0.1cm}

If $J\in\calD$ and $|J|\ge1/4$, then by \textit{Claim 2}, we have
$$
\int_{Q_J}\left|\sum_{I\in\calD}c_I\Phi_I'(z)\right|^2(1-|z|^2)^p\,dA(z)
\lesssim\sum_{I\in\calD}|c_I|^2|I|^p
\le1\lesssim |J|^p.
$$

Finally, for an arbitrary arc $J$, we may cover it by at most two dyadic arcs whose lengths lie between $|J|$ and $2|J|$. Since their tents cover $Q_J$, we obtain the desired bound by applying the estimates already proved for dyadic arcs. The proof of \textit{Claim 3} is complete, and therefore we have proved (1) for finitely supported sequences. 

\medskip 

We now prove (1) for a general sequence $(c_I)_{I\in\calD}$, which need not be finitely supported. We first assume that
$$
\sup_{J\in\calD}\frac{1}{|J|^p}
\sum_{\substack{I\in\calD\\I\subseteq J}}|c_I|^2|I|^p\le1.
$$
For each $N\ge0$, the truncated sequence $(c_I)_{I\in\calD_{\le N}}$, extended by zero outside $\calD_{\le N}$, satisfies the same bound. Since $|c_I|\le1$, by Theorem~\ref{20261017thm01}, (1), the series
$$
f(z):=\sum_{n\ge0}\sum_{I\in\calD_n}c_I\Phi_I(z)
$$
converges absolutely and uniformly on compact subsets of $\D$ and defines a function $f\in\mathring{\calB}$. By Cauchy's integral formula, we also have
$$
\sum_{I\in\calD_{\le N}}c_I\Phi_I'(z)\longrightarrow f'(z)
$$
uniformly on compact subsets of $\D$. Thus, for every arc $J\subseteq\T$, by Fatou's lemma and the estimate already established for finitely supported sequences, we obtain
\begin{align*}
\int_{Q_J}|f'(z)|^2(1-|z|^2)^p\,dA(z)
&\le \liminf_{N\to\infty}
\int_{Q_J}\left|\sum_{I\in\calD_{\le N}}c_I\Phi_I'(z)\right|^2(1-|z|^2)^p\,dA(z)\\
&\lesssim |J|^p.
\end{align*}
By \eqref{20261024eq01}, we conclude that $\|f\|_{\calQ_p,*}\lesssim1$. The general estimate in (1) follows by homogeneity.

\vspace{0.1cm}

To prove (2), let $f\in\calQ_p$ with $f(0)=0$. Since $\calQ_p\subseteq\calB$, we use the representation in Theorem~\ref{20261017thm01}, (2), with
$$
c_\T=0,\qquad \textrm{and} \qquad c_I=f_n(\zeta_I),\quad I\in\calD_n,\quad n\ge1,
$$
where $f_n$ is defined in \eqref{20261018eq02}. Since $f_0=f_1=0$, it remains to establish the following coefficient estimate.

\vspace{0.1cm}

\noindent\textit{Claim 4.} For every $J\in\calD_m$,
$$
\sum_{n\ge\max(m,2)}2^{-np}
\left(\sum_{\substack{I\in\calD_n\\I\subseteq J}}|f_n(\zeta_I)|^2\right)
\lesssim |J|^p\|f\|_{\calQ_p,*}^2.
$$

\begin{proof}[Proof of Claim 4]
We first estimate $f_n(\zeta_I)$ in terms of $f'$. Recall that, in the proof of Lemma~\ref{20261017lem20}, we defined
$$
q_n(s)=\frac{\varphi(s)}{s(1-2^{-n})^{2^ns}},
\qquad n\ge2,
$$
with $q_n$ extended by zero outside the support of $\varphi$. By direct computation, we have
$$
q_n^{(4)}(s)
=(1-2^{-n})^{-2^ns}
\sum_{j=0}^4\binom{4}{j}
\bigl[-2^n\log(1-2^{-n})\bigr]^{4-j}
\frac{d^j}{ds^j}\left(\frac{\varphi(s)}s\right).
$$
Since $\varphi$ is supported in $[1/8,1/2]$, we obtain
$$
\|q_n\|_\infty+\|q_n^{(4)}\|_\infty\lesssim1.
$$
Applying \eqref{20261024eq05a}, we have
\begin{equation} \label{20261025eq01}
|K_{2^n,q_n}(e^{2\pi i(t-s)})|
\lesssim\frac{2^n}{(1+2^nd_\T(t,s))^4},
\end{equation} 
which gives
\begin{equation} \label{20261025eq02}
\int_0^1|K_{2^n,q_n}(e^{2\pi i(t-s)})|\,ds
\lesssim\int_0^{1/2}\frac{2^n}{(1+2^ns)^4}\,ds
\lesssim1.
\end{equation} 
Returning to the estimate of $f_n(\zeta_I)$, by \eqref{20261025eq03}, we have
$$
f_n(\zeta_I)=2^{-n}(1-2^{-n})\int_0^1
K_{2^n,q_n}(e^{2\pi i(t_I-s)})e^{2\pi is}
f'((1-2^{-n})e^{2\pi is})\,ds.
$$
Thus, by \eqref{20261025eq01} and \eqref{20261025eq02}, we obtain
\begin{align} \label{20261025eq10}
|f_n(\zeta_I)|^2
&\le2^{-2n}
\left(\int_0^1|K_{2^n,q_n}(e^{2\pi i(t_I-s)})|\,ds\right) \nonumber \\
&\qquad\cdot
\left(\int_0^1|K_{2^n,q_n}(e^{2\pi i(t_I-s)})|
|f'((1-2^{-n})e^{2\pi is})|^2\,ds\right) \nonumber \\
&\lesssim2^{-n}\int_0^1
\frac{|f'((1-2^{-n})e^{2\pi is})|^2}{(1+2^nd_\T(t_I,s))^4}\,ds.
\end{align}
For each $s \in [0, 1)$, by subharmonicity, we have
$$
|f'((1-2^{-n})e^{2\pi is})|^2
\lesssim2^{2n}
\int_{\{|z-(1-2^{-n})e^{2\pi is}|<2^{-n}/8\}}|f'(z)|^2\,dA(z).
$$
Also note that, for any $s\in[0,1)$, we have
$$
\left\{z\in\D:\left|z-(1-2^{-n})e^{2\pi is}\right|<\frac{2^{-n}}8\right\}
\subseteq
\left\{z\in\D:2^{-n-1}<1-|z|<3\cdot2^{-n-1}\right\}.
$$
Moreover, if $|z-(1-2^{-n})e^{2\pi is}|<2^{-n}/8$, then
$$
d_\T\left(s,\frac{\arg z}{2\pi}\right)\lesssim2^{-n}.
$$
Consequently, we have
$$
1+2^nd_\T(t_I,s)
\simeq1+2^nd_\T\left(t_I,\frac{\arg z}{2\pi}\right).
$$
This gives
$$
\int_{\substack{0\le s<1\\|z-(1-2^{-n})e^{2\pi is}|<2^{-n}/8}}
\frac{ds}{(1+2^nd_\T(t_I,s))^4}
\lesssim
\frac{2^{-n}}{(1+2^nd_\T(t_I,\arg z/(2\pi)))^4}.
$$
Plugging the above estimate back to \eqref{20261025eq10}, we deduce that 
\begin{align*}
|f_n(\zeta_I)|^2
&\lesssim2^n
\int_{\{2^{-n-1}<1-|z|<3\cdot2^{-n-1}\}}|f'(z)|^2\\
&\qquad\cdot
\left(\int_{\substack{0\le s<1\\|z-(1-2^{-n})e^{2\pi is}|<2^{-n}/8}}
\frac{ds}{(1+2^nd_\T(t_I,s))^4}\right)dA(z)\\
&\lesssim
\int_{\{2^{-n-1}<1-|z|<3\cdot2^{-n-1}\}}
\frac{|f'(z)|^2}{(1+2^nd_\T(t_I,\arg z/(2\pi)))^4}\,dA(z), 
\end{align*}
which further implies
\begin{equation}\label{20261024eq07}
2^{-np} \cdot |f_n(\zeta_I)|^2 
\lesssim
\int_{\{2^{-n-1}<1-|z|<3\cdot2^{-n-1}\}}
\frac{|f'(z)|^2(1-|z|^2)^p}{(1+2^nd_\T(t_I,\arg z/(2\pi)))^4}\,dA(z).
\end{equation}

Now fix $J\in\calD_m$. For $n\ge m$, we have
\begin{align*}
\sum_{\substack{I\in\calD_n\\I\subseteq J}}
\frac{1}{(1+2^nd_\T(t_I,t))^4}
&\lesssim
\sum_{j\ge\lfloor2^n\operatorname{dist}_\T(t,J)\rfloor}\frac{1}{(1+j)^4}\\
&\lesssim(1+2^n\operatorname{dist}_\T(t,J))^{-3},
\end{align*}
where $\operatorname{dist}_\T(t,J)$ denotes the distance from $e^{2\pi it}$ to $J$ in normalized angular coordinates.

Since $|J| \ge 2^{-n}$ for $n\ge m$, we have
$$
(1+2^n\operatorname{dist}_\T(t,J))^{-3}
\le\left(1+\frac{\operatorname{dist}_\T(t,J)}{|J|}\right)^{-3}.
$$
Also, every point $z\in\D$ belongs to at most two of the annuli in \eqref{20261024eq07}. Since $|J|=2^{-m}$, we also have
$$
\bigcup_{n\ge\max(m,2)}
\left\{z\in\D:2^{-n-1}<1-|z|<3\cdot2^{-n-1}\right\}
\subseteq
\left\{z\in\D:1-|z|<\frac{3|J|}{2}\right\}.
$$
Summing \eqref{20261024eq07} over the relevant $I$ and $n$, we therefore obtain
\begin{align} \label{20261025eq30}
&\sum_{n\ge\max(m,2)}2^{-np}
\left(\sum_{\substack{I\in\calD_n\\I\subseteq J}}|f_n(\zeta_I)|^2\right) \nonumber \\
&\qquad\lesssim
\int_{\{1-|z|<3|J|/2\}}
\frac{|f'(z)|^2(1-|z|^2)^p}
{\left(1+\operatorname{dist}_\T(\arg z/(2\pi),J)/|J|\right)^3}\,dA(z).
\end{align}

To estimate the last integral above, we separate the integration region according to the angular variable. By \eqref{20261024eq01}, we then obtain
\begin{align*}
\eqref{20261025eq30} 
& \le \int_{Q_{3J}}|f'(z)|^2(1-|z|^2)^p\,dA(z)\\
&\qquad\quad+
\sum_{j\ge1}\int_{\substack{1-|z|<3|J|/2\\
e^{i\arg z}\in3\cdot2^jJ\setminus3\cdot2^{j-1}J}}
\frac{|f'(z)|^2(1-|z|^2)^p}
{\left(1+\operatorname{dist}_\T(\arg z/(2\pi),J)/|J|\right)^3}\,dA(z)\\
&\lesssim \int_{Q_{3J}}|f'(z)|^2(1-|z|^2)^p\,dA(z)
+\sum_{j\ge1}2^{-3j}
\left(\int_{Q_{3\cdot2^jJ}}|f'(z)|^2(1-|z|^2)^p\,dA(z)\right)\\
&\le \|f\|_{\calQ_p,*}^2
\left(|3J|^p+\sum_{j\ge1}2^{-3j}|3\cdot2^jJ|^p\right)\\
&\lesssim |J|^p\|f\|_{\calQ_p,*}^2
\left(1+\sum_{j\ge1}2^{-j(3-p)}\right)
\lesssim |J|^p\|f\|_{\calQ_p,*}^2.
\end{align*}
Here, in the second estimate above, we have used the following facts:
\begin{enumerate}
    \item For $j\ge1$ and $e^{i\arg z}\in3\cdot2^jJ\setminus3\cdot2^{j-1}J$, we have
    $$
    \operatorname{dist}_\T\left(\frac{\arg z}{2\pi},J\right)
    \ge \frac{3\cdot2^{j-1}|J|-|J|}{2}
    \ge 2^{j-1}|J|.
    $$
    Hence,
    $$
    \frac{1}{\left(1+\operatorname{dist}_\T(\arg z/(2\pi),J)/|J|\right)^3}
    \le \frac{1}{(1+2^{j-1})^3}
    \lesssim 2^{-3j}.
    $$

    \item For every $j\ge1$, we have
    $$
    \left\{z\in\D:1-|z|<\frac{3|J|}{2},\quad
    e^{i\arg z}\in3\cdot2^jJ\setminus3\cdot2^{j-1}J\right\}
    \subseteq Q_{3\cdot2^jJ}.
    $$
\end{enumerate}
This proves Claim 4.
\end{proof}

Finally, for $J\in\calD_m$, we have
$$
\sum_{\substack{I\in\calD\\I\subseteq J}}|c_I|^2|I|^p
=\sum_{n\ge\max(m,2)}2^{-np}
\left(\sum_{\substack{I\in\calD_n\\I\subseteq J}}|f_n(\zeta_I)|^2\right)
\lesssim |J|^p\|f\|_{\calQ_p,*}^2.
$$
This proves (2). The proof is complete. 
\end{proof}

\medskip 

\subsection{Proof of Theorem~\ref{20261024thm01}}

\subsubsection{Boundedness}

By Proposition~\ref{20261016prop01} and $\|f\|_{\calB,*}\lesssim\|f\|_{\calQ_p,*}$, we obtain
$$
\|\operatorname{id}:\calQ_p\to L^q(\mu)\|
\simeq\|\operatorname{id}:\calQ_p\to L^q(\mu_\calD)\|,
$$
where $\mu_\calD$ is defined in \eqref{20261016eq10}.

To apply Lemma~\ref{20261024lem01} to this discrete embedding, for each $N\ge0$, we consider coefficient vectors $c=(c_I)_{I\in\calD_{\le N}}$ satisfying
\begin{equation}\label{20261024eq08}
\sum_{\substack{I\in\calD_{\le N}\\I\subseteq J}}|c_I|^2|I|^p\le |J|^p,
\qquad J\in\calD_{\le N},
\end{equation}
and define
$$
C_{\calD,p,q}(\mu):=
\sup_{N\ge0}\ \sup_{\substack{c\text{ satisfying}\\\eqref{20261024eq08}}}
\left(\sum_{K\in\calD}\mu(Q_K^{\mathrm{up}})
\left|\sum_{I\in\calD_{\le N}}c_I\Phi_I(z_K)\right|^q\right)^{2/q}.
$$

Suppose first that $\operatorname{id}:\calQ_p\to L^q(\mu)$ is bounded, and hence $id: \calQ_p \to L^q(\mu_{\calD})$ is also bounded. By Lemma~\ref{20261024lem01}, (1), and the fact that $\Phi_I(0)=0$ for every $I\in\calD$, we have
$$
\left\|\sum_{I\in\calD_{\le N}}c_I\Phi_I\right\|_{\calQ_p}
=\left\|\sum_{I\in\calD_{\le N}}c_I\Phi_I\right\|_{\calQ_p,*}
\lesssim1.
$$
Applying the bounded embedding into $L^q(\mu_\calD)$ to the functions $f_N(z):=\sum_{I \in \calD_{\le N}} c_I \Phi_I$ and taking the supremum over $c$ and $N$, we obtain
$$
C_{\calD,p,q}(\mu)
\lesssim\|\operatorname{id}:\calQ_p\to L^q(\mu_\calD)\|^2<\infty.
$$

Conversely, suppose that $C_{\calD,p,q}(\mu)<\infty$. For $f\in\calQ_p$ with $f(0)=0$, we choose coefficients as in Lemma~\ref{20261024lem01}, (2). Since truncation preserves the coefficient bounds, by the definition of $C_{\calD,p,q}(\mu)$ and Fatou's lemma, we obtain
$$
\sum_{K\in\calD}\mu(Q_K^{\mathrm{up}})|f(z_K)|^q
\lesssim C_{\calD,p,q}(\mu)^{q/2}\|f\|_{\calQ_p,*}^q.
$$
For a general $f\in\calQ_p$, we apply the above estimate to $f-f(0)$. Since $\mu(\D)<\infty$, this proves that $\operatorname{id}:\calQ_p\to L^q(\mu_\calD)$, and hence $\operatorname{id}:\calQ_p\to L^q(\mu)$, is bounded. This together with the lower bound obtained by testing the constant function $1$, yield
\begin{equation}\label{20261024eq09}
\|\operatorname{id}:\calQ_p\to L^q(\mu)\|
\simeq \mu(\D)^{1/q}+C^{1/2}_{\calD,p,q}(\mu).
\end{equation}

It remains to prove
$$
C_{\calD,p,q}(\mu)\simeq\mathfrak{Q}_{\calD,p,q}(\mu).
$$
We begin with analyzing the innermost infimum in the definition of $\mathfrak{Q}_{\calD,p,q}(\mu)$. Fix $N\ge0$ and define
$$
\mathcal M_{p,N}:=
\left\{X\ge0:
\sum_{\substack{I\in\calD_{\le N}\\I\subseteq J}}|I|^pX_{II}\le |J|^p,
\quad J\in\calD_{\le N}\right\},
$$
where the matrices are indexed by $\calD_{\le N}$. We claim that
\begin{equation}\label{20261024eq11}
\begin{aligned}
&\inf\left\{
\sum_{J\in\calD_{\le N}}\lambda_J|J|^p:
\begin{array}{l}
\lambda_J\ge0,\quad J\in\calD_{\le N},\\[2pt]
\displaystyle\Gamma_{\mu,h,N}\le
\operatorname{diag}_{I\in\calD_{\le N}}
\left(|I|^p\sum_{\substack{J\in\calD_{\le N}\\I\subseteq J}}\lambda_J\right)
\end{array}
\right\}=
\sup_{X\in\mathcal M_{p,N}}
\sum_{K\in\calD}\mu(Q_K^{\mathrm{up}})h_Ka_K(X),
\end{aligned}
\end{equation}
where $a_K(X)$ is defined in \eqref{20261019eq06}.

Suppose first that all entries of $\Gamma_{\mu,h,N}$ are finite. By a direct computation, we have
\begin{align*}
\operatorname{Tr}\left[
X\operatorname{diag}_{I\in\calD_{\le N}}
\left(|I|^p\sum_{\substack{J\in\calD_{\le N}\\I\subseteq J}}\lambda_J\right)
\right]
&=\sum_{I\in\calD_{\le N}}X_{II}|I|^p
\left(\sum_{\substack{J\in\calD_{\le N}\\I\subseteq J}}\lambda_J\right)\\
&= \sum_{J\in\calD_{\le N}}\lambda_J
\left(\sum_{\substack{I\in\calD_{\le N}\\I\subseteq J}}|I|^pX_{II}\right).
\end{align*}
Thus we may write the corresponding Lagrangian as
$$
\operatorname{Tr}(\Gamma_{\mu,h,N}X)
+\sum_{J\in\calD_{\le N}}\lambda_J
\left(|J|^p-\sum_{\substack{I\in\calD_{\le N}\\I\subseteq J}}|I|^pX_{II}\right).
$$
Taking its infimum over $\lambda_J\ge0$, we obtain $\operatorname{Tr}(\Gamma_{\mu,h,N}X)$ when $X\in\mathcal M_{p,N}$ and $-\infty$ otherwise. Moreover, by choosing all $\lambda_J$ sufficiently large, we verify Slater's condition. By finite-dimensional semidefinite duality and the trace identity established in the proof of Lemma~\ref{20261019lem02}, we obtain \eqref{20261024eq11}.

If a diagonal entry is infinite, we choose the matrix with that diagonal entry equal to $1$ and all other entries zero. This matrix belongs to $\mathcal M_{p,N}$, and we obtain $+\infty$ on the right-hand side of \eqref{20261024eq11}, in agreement with the convention for the innermost infimum.

We now evaluate the optimization over $h$ in the three ranges of $q$.
\begin{enumerate}
\item For $0<q<2$, we first prove that
$$
\begin{aligned}
&\inf_{\substack{h_K>0\\
\sum_{K\in\calD}\mu(Q_K^{\mathrm{up}})h_K^{q/(q-2)}\le1}}
\sup_{X\in\mathcal M_{p,N}}
\sum_{K\in\calD}\mu(Q_K^{\mathrm{up}})h_Ka_K(X)\\
&\qquad\le
\sup_{X\in\mathcal M_{p,N}}
\left(\sum_{K\in\calD}\mu(Q_K^{\mathrm{up}})a_K(X)^{q/2}\right)^{2/q}.
\end{aligned}
$$
The case $\mu=0$ is trivial, so we may assume that $\mu(\D)>0$. Observe that $\mathcal M_{p,N}$ is nonempty, closed, and convex. Indeed, it contains zero, and its defining constraints are closed and preserved under convex combinations. Moreover, taking $J=I$ in these constraints, we obtain $X_{II}\le1$, and hence $|X_{IJ}|\le1$ by positivity. Thus $\mathcal M_{p,N}$ is bounded and therefore compact. Repeating the argument used to prove \eqref{20261020eq30}, with $\mathcal M_N$ replaced by $\mathcal M_{p,N}$ and the exponent replaced by $q$, proves the above estimate.

On the other hand, let $h=(h_K)_{K\in\calD}$ be a positive sequence satisfying
$$
\sum_{K\in\calD}\mu(Q_K^{\mathrm{up}})h_K^{q/(q-2)}\le1.
$$
For every $X\in\mathcal M_{p,N}$, by H\"older's inequality, we have
\begin{align*}
\left(\sum_{K\in\calD}\mu(Q_K^{\mathrm{up}})a_K(X)^{q/2}\right)^{2/q}
&\le
\left(\sum_{K\in\calD}\mu(Q_K^{\mathrm{up}})h_Ka_K(X)\right)
\left(\sum_{K\in\calD}\mu(Q_K^{\mathrm{up}})h_K^{q/(q-2)}\right)^{(2-q)/q}\\
&\le \sum_{K\in\calD}\mu(Q_K^{\mathrm{up}})h_Ka_K(X).
\end{align*}
Taking the supremum over $X\in\mathcal M_{p,N}$ and then the infimum over all positive sequences $h$ satisfying the above constraint, we obtain
$$
\begin{aligned}
&\sup_{X\in\mathcal M_{p,N}}
\left(\sum_{K\in\calD}\mu(Q_K^{\mathrm{up}})a_K(X)^{q/2}\right)^{2/q}\\
&\qquad\le
\inf_{\substack{h_K>0\\
\sum_{K\in\calD}\mu(Q_K^{\mathrm{up}})h_K^{q/(q-2)}\le1}}
\sup_{X\in\mathcal M_{p,N}}
\sum_{K\in\calD}\mu(Q_K^{\mathrm{up}})h_Ka_K(X).
\end{aligned}
$$
Combining these two estimates establishes the desired identity.

    \item For $q=2$, we take $h_K=1$ in \eqref{20261024eq11}.

    \item For $q>2$, we may interchange the two supremum over $h$ and $X$. For each fixed $X$, by duality, we have
    $$
    \sup_{\substack{h_K\ge0\\
    \sum_{K\in\calD}\mu(Q_K^{\mathrm{up}})h_K^{q/(q-2)}\le1}}
    \sum_{K\in\calD}\mu(Q_K^{\mathrm{up}})h_Ka_K(X)
    =\left(\sum_{K\in\calD}\mu(Q_K^{\mathrm{up}})a_K(X)^{q/2}\right)^{2/q}.
    $$
\end{enumerate}
Consequently, for every $0<q<\infty$, we obtain
$$
\mathfrak{Q}_{\calD,p,q}(\mu)
=\sup_{N\ge0}\sup_{X\in\mathcal M_{p,N}}
\left(\sum_{K\in\calD}\mu(Q_K^{\mathrm{up}})a_K(X)^{q/2}\right)^{2/q}.
$$

Therefore, we are left to show that
\begin{equation}\label{20261024eq10}
C_{\calD,p,q}(\mu)
\simeq
\sup_{N\ge0}\sup_{X\in\mathcal M_{p,N}}
\left(\sum_{K\in\calD}\mu(Q_K^{\mathrm{up}})a_K(X)^{q/2}\right)^{2/q}.
\end{equation}
For any $c$ satisfying \eqref{20261024eq08}, we may take $X=cc^*\in\mathcal M_{p,N}$, for which
$$
a_K(X)=\left|\sum_{I\in\calD_{\le N}}c_I\Phi_I(z_K)\right|^2.
$$
Taking the supremum over all $c$ satisfying \eqref{20261024eq08} and over $N\ge0$, we obtain
$$
C_{\calD,p,q}(\mu)\le
\sup_{N\ge0}\ \sup_{X\in\mathcal M_{p,N}}
\left(\sum_{K\in\calD}\mu(Q_K^{\mathrm{up}})a_K(X)^{q/2}\right)^{2/q}.
$$

For the reverse estimate, we may assume that $C_{\calD,p,q}(\mu)<\infty$. Fix $N\ge0$ and $X\in\mathcal M_{p,N}$. We consider the linear map
\begin{equation}\label{20261025eq60}
(\alpha_I)_{I\in\calD_{\le N}}
\longmapsto
\sum_{I\in\calD_{\le N}}\sqrt{X_{II}}\,\alpha_I\Phi_I
\end{equation}
from $\ell^\infty(\calD_{\le N})$ into $L^q(\mu_\calD)$. For $|\alpha_I|\le1$, we have
$$
\sum_{\substack{I\in\calD_{\le N}\\I\subseteq J}}|I|^pX_{II}|\alpha_I|^2
\le\sum_{\substack{I\in\calD_{\le N}\\I\subseteq J}}|I|^pX_{II}
\le |J|^p,
\qquad J\in\calD_{\le N}.
$$
Thus the coefficients $(\sqrt{X_{II}}\,\alpha_I)_{I\in\calD_{\le N}}$ satisfy \eqref{20261024eq08}. By the definition of $C_{\calD,p,q}(\mu)$, the operator norm of the map in \eqref{20261025eq60} is therefore at most $C_{\calD,p,q}(\mu)^{1/2}$.

Since $X\ge0$, we may write
$$
X=\sum_{j=1}^s c^{(j)}(c^{(j)})^*.
$$
For $1\le j\le s$, we define
\begin{equation}\label{20261025eq61}
\alpha^{(j)}=(\alpha_I^{(j)})_{I\in\calD_{\le N}},
\qquad
\alpha_I^{(j)}:=
\begin{cases}
c_I^{(j)}/\sqrt{X_{II}},&X_{II}>0,\\
0,&X_{II}=0.
\end{cases}
\end{equation}
Since $\sum_{j=1}^s|c_I^{(j)}|^2=X_{II}$, we have
$$
\max_{I\in\calD_{\le N}}\sum_{j=1}^s|\alpha_I^{(j)}|^2\le1,
\qquad
\sqrt{X_{II}}\,\alpha_I^{(j)}=c_I^{(j)}.
$$

We now repeat the argument used to prove \eqref{20261019eq01}, applying it to the map \eqref{20261025eq60} and the vectors defined in \eqref{20261025eq61}, with target exponent $q$. We obtain
\begin{align*}
\left(\sum_{K\in\calD}\mu(Q_K^{\mathrm{up}})a_K(X)^{q/2}\right)^{2/q}
&=\left[
\sum_{K\in\calD}\mu(Q_K^{\mathrm{up}})
\left(\sum_{j=1}^s
\left|\sum_{I\in\calD_{\le N}}\sqrt{X_{II}}\,\alpha_I^{(j)}\Phi_I(z_K)\right|^2
\right)^{q/2}
\right]^{2/q}\\
&\lesssim C_{\calD,p,q}(\mu)
\max_{I\in\calD_{\le N}}\sum_{j=1}^s|\alpha_I^{(j)}|^2\\
&\le C_{\calD,p,q}(\mu).
\end{align*}
Taking the supremum over $X\in\mathcal M_{p,N}$ and $N\ge0$, we conclude that
$$
\sup_{N\ge0}\ \sup_{X\in\mathcal M_{p,N}}
\left(\sum_{K\in\calD}\mu(Q_K^{\mathrm{up}})a_K(X)^{q/2}\right)^{2/q}
\lesssim C_{\calD,p,q}(\mu).
$$
Together with the inequality already established, this proves \eqref{20261024eq10}. Combining this with \eqref{20261024eq09}, we conclude the boundedness assertion and \eqref{20261024eq03}.

\subsubsection{Compactness}
The proof of the compactness assertion is similar to that of Theorem~\ref{thm:compact}.
Applying \eqref{20261024eq03} to $\one_{S_\rho}\mu$, we obtain
\begin{equation}\label{20261025eq62}
\|\operatorname{id}:\calQ_p\to L^q(\one_{S_\rho}\mu)\|^2
\simeq \mu(S_\rho)^{2/q}+\mathfrak{Q}_{\calD,p,q}(\one_{S_\rho}\mu).
\end{equation}
Since $\mu$ is finite, we have $\mu(S_\rho)\to0$ as $\rho\to1^-$. Thus
$$
\lim_{\rho\to1^-}\mathfrak{Q}_{\calD,p,q}(\one_{S_\rho}\mu)=0
$$
if and only if
\begin{equation}\label{20261025eq63}
\lim_{\rho\to1^-}\|\operatorname{id}:\calQ_p\to L^q(\one_{S_\rho}\mu)\|=0.
\end{equation}

 It remains to show that, under the boundedness assumption, the embedding $\operatorname{id}:\calQ_p\to L^q(\mu)$ is compact if and only if \eqref{20261025eq63} holds. This follows by repeating the proof of Theorem~\ref{thm:compact}, with $\calB$ replaced by $\calQ_p$, the target exponent replaced by $q$, and \eqref{20261016eq05} replaced by \eqref{20261024eq03}. Together with \eqref{20261025eq62} and the boundedness characterization, this completes the proof of Theorem~\ref{20261024thm01}.

For $p=1$, we obtain the corresponding characterization for $\mathrm{BMOA}$, since $\calQ_1=\mathrm{BMOA}$. For $p>1$, we have
$$
\sup_{J\in\calD}\frac{1}{|J|^p}
\sum_{\substack{I\in\calD\\I\subseteq J}}|c_I|^2|I|^p
\simeq\sup_{I\in\calD}|c_I|^2.
$$
Indeed, for every $J\in\calD$, we have
\begin{align*}
\frac{1}{|J|^p}\sum_{\substack{I\in\calD\\I\subseteq J}}|c_I|^2|I|^p
&\le\left(\sup_{I\in\calD}|c_I|^2\right)
\frac{1}{|J|^p}\sum_{n\ge0}2^n(2^{-n}|J|)^p\\
&=\left(\sup_{I\in\calD}|c_I|^2\right)\sum_{n\ge0}2^{-n(p-1)}
\lesssim\sup_{I\in\calD}|c_I|^2.
\end{align*}
For the reverse estimate, taking $J=I$, we obtain
$$
|c_I|^2\le
\sup_{J\in\calD}\frac{1}{|J|^p}
\sum_{\substack{K\in\calD\\K\subseteq J}}|c_K|^2|K|^p,
\qquad I\in\calD.
$$
Taking the supremum over $I$ proves the desired comparison. Thus, for $p>1$, the coefficient condition reduces to boundedness of the coefficients, in agreement with $\calQ_p=\calB$.

\begin{rem} \label{20261025rem01}
Under the hypotheses on $K$ in \cite{HLXZ}, the corresponding extension to $\calQ_K$ follows by a straightforward modification of the coefficient constraints, using the weighted tent estimates established there. We therefore do not pursue this direction here and leave the details to the interested reader.
\end{rem}

\bigskip
\noindent\textbf{Acknowledgments.}
The authors would like to thank Ruhan Zhao for bringing the
Bloch--Carleson measure problem to their attention and for helpful
discussions. They also thank Jos\'e Pel\'aez and Jouni R\"atty\"a for
valuable comments and suggestions. The first author was supported by the NSF grant DMS-2555999 and by the
Simons Travel grant MPS-TSM-00007213.

\end{document}